%% file: chaumontfrelet_2022a.tex
\documentclass[a4paper,11pt]{amsart} 

\usepackage{a4wide}
\usepackage{pgfplots}
\usepackage{subcaption}
\usepackage{rotating}

\usepackage[foot]{amsaddr}

\usepackage[breaklinks,bookmarks=false]{hyperref}

\hypersetup{colorlinks,linkcolor=blue,citecolor=blue,urlcolor=blue,plainpages=false,pdfwindowui=false}

\input{general_commands}

\input{special_commands}

\input{slope_triangle}

\newcommand{\hu}{\widehat u}

\newcommand{\hf}{\widehat f}
\newcommand{\hg}{\widehat g}
\newcommand{\hxi}{\widehat \xi}
\newcommand{\bsig}{\boldsymbol \sigma}
\newcommand{\btau}{\boldsymbol \tau}

\newcommand{\hgamma}{\widehat \gamma}

\newcommand{\RT}{\boldsymbol{RT}}
\newcommand{\CTa}{\CT_{\ba}}
\newcommand{\CFa}{\CF_{\ba}}

\newtheorem{theorem}{Theorem}
\newtheorem{lemma}[theorem]{Lemma}
\newtheorem{corollary}[theorem]{Corollary}

\newtheorem{proposition}[theorem]{Proposition}
\newtheorem{definition}[theorem]{Definition}
\newtheorem{remark}[theorem]{Remark}

\newcommand{\enorm}[1]{|\!|\!|#1|\!|\!|}

\newcommand{\GD}{{\Gamma_{\rm D}}}
\newcommand{\GA}{{\Gamma_{\rm A}}}

\numberwithin{equation}{section}
\numberwithin{figure}{section}
\numberwithin{theorem}{section}

\title[Constant-free and p-robust estimates for the wave equation]%
{Asymptotically constant-free and polynomial-degree-robust a posteriori
estimates for space discretizations of the wave equation}

\address{\vspace{-.5cm}}
\address{\noindent \tiny \textup{$^\dagger$Inria Universit\'e C\^ote d'Azur, LJAD, CNRS}}

\author{T. Chaumont-Frelet$^\dagger$}

\begin{document}

\maketitle

\begin{abstract}
We derive an equilibrated a posteriori error estimator for the space (semi) discretization
of the scalar wave equation by finite elements. In the idealized setting where time discretization
is ignored and the simulation time is large, we provide fully-guaranteed upper bounds
that are asymptotically constant-free and show that the proposed estimator is
efficient and polynomial-degree-robust, meaning that the efficiency constant does
not deteriorate as the approximation order is increased. To the best of our knowledge,
this work is the first to derive provably efficient error estimates for the wave equation.
We also explain, without analysis, how the estimator is adapted to cover time discretization
by an explicit time integration scheme. Numerical examples illustrate the theory and suggest
that it is sharp.

\vspace{.5cm}
\end{abstract}


\section{Introduction}

Given a domain $\Omega$, a partition of its boundary into disjoint sets $\GD$ and $\GA$
and space-variable coefficients $\mu,\BA$ and $\gamma$, the scalar wave equation consists
in finding $u$ such that
\begin{equation}
\label{eq_wave_equation}
\left \{
\begin{array}{rcll}
\mu \ddot u - \div \left (\BA\grad u\right ) &=& \mu f & \text{ in } \Omega,
\\
\gamma \dot u + \BA \grad u \cdot \bn &=& \gamma g & \text{ on } \GA,
\\
u &=& 0 & \text{ on } \GD,
\end{array}
\right .
\end{equation}
$u|_{t=0} = u_0$ and $\dot u|_{t=0} = u_1$, where $f,g,u_0$ and $u_1$ are given
right-hand sides and initial conditions. This problem models the propagation of
(small) acoustics waves \cite{jacobsen_juhl_2013a}, and is a good mathematical simplification
to later address more complicated physical waves arising for instance in elastodynamics and
electromagnetism. As a result, the analysis of \eqref{eq_wave_equation} and its
numerical discretization has attracted much attention over the past decades
\cite{%
baldassari_barucq_calandra_denel_diaz_2012a,%
bernardi_suli_2005a,%
burman_duran_ern_steins_2021a,%
cohen_joly_roberts_tordjman_2001a,%
fezoui_lanteri_lohrengel_piperno_2005a,%
georgoulis_lakkis_makridakis_2013a,%
gorynina_lozinski_picasso_2019a,%
griesmaier_monk_2014a,%
grote_schneebeli_schotzau_2006a,%
hesthaven_warburton_2002a,%
picasso_2010a}.

Here, we are interested in finite element discretizations of
\eqref{eq_wave_equation}, which have the advantage to easily handle complex geometries.
For the sake of simplicity we focus on conforming Lagrange elements
\cite{boffi_brezzi_fortin_2013a,cohen_joly_roberts_tordjman_2001a},
but discontinuous Galerkin
\cite{%
fezoui_lanteri_lohrengel_piperno_2005a,%
grote_schneebeli_schotzau_2006a,%
hesthaven_warburton_2002a}
or hybridized
\cite{burman_duran_ern_steins_2021a,griesmaier_monk_2014a}
methods could be considered as well. As compared to elliptic and parabolic problems (see, e.g,
\cite{ainsworth_oden_2000a,ern_vohralik_2015a} and
\cite{ern_vohralik_2010a,makridakis_nochetto_2003a,verfurth_2013a}, and the references therein),
the a posteriori error analysis of \eqref{eq_wave_equation} has not received much attention
and it is the focus of the present work.

To the best of the author's knowledge, the first references considering a posteriori
error estimation for the wave equation are \cite{adjerid_2002a,adjerid_2006a}.
There, an asymptotically exact error estimator is proposed for the space semi discretization.
The analysis hinges on hierarchical polynomial basis, and strongly relies on tensor-product
meshes and the smoothness of the solution. Hence, this methodology is not suited for
general geometries which requires simplicial (or general hexahedral) meshes and where the
solution may exhibit singularities.

In \cite{bernardi_suli_2005a}, an estimator is proposed for the
full discretization of \eqref{eq_wave_equation}, where an implicit
Euler scheme is used for time integration. Lower and upper bounds for
the error are derived, but these bounds are not entirely satisfactory.
Indeed, the norms employed to measure the error in the upper and lower
bounds are different (compare \cite[Corollary 5.3]{bernardi_suli_2005a} and
\cite[Theorem 5.1]{bernardi_suli_2005a}), indicating that the estimator may
not be efficient. Besides, since this approach explicitly relies on the implicit
Euler scheme, it is unclear whether it can be extended to space  semi-discretzation
or to explicit time stepping. Finally, the behaviour of the constants in the estimates
with respect to the problem parameters is unclear.

Another family of estimators, relying on the an elliptic reconstriction
presented and analyzed in
\cite{georgoulis_lakkis_makridakis_2013a,gorynina_lozinski_picasso_2019a,picasso_2010a}.
Upper bounds for the error measured in various norms,
including the ``$L^\infty(L^2)$'' and ``$L^2(H^1)$'' norms, have been derived,
but lower bounds are not available so far.

Goal-oriented adaptive algorithms are also derived in
\cite{bangerth_geiger_rannacher_2010a,bangerth_rannacher_2001a},
where the focus is more on the algorithmic design and numerical tests
than on theoretical analysis.

In this work, we discretize \eqref{eq_wave_equation} with conforming Lagrange finite elements
in space. We neglect the time-discretization, and study the resulting semi-discrete problem.
In spirit, the design of our estimator is similar to the one in \cite{bernardi_suli_2005a},
as it is directly constructed using the ``instantaneous residual'' of the PDE
doe not rely on an elliptic reconstruction. However, in contrast with \cite{bernardi_suli_2005a},
our analysis does not rely on the use of an implicit scheme and, although we focus on the
semi-discrete case here, it could be extended to cover an explicit time scheme. Another key
feature of our analysis is that we obtain constants that are much more explicit, in the upper
bound and in the lower bound, than in \cite{bernardi_suli_2005a}. Actually, although our
methodology is general, we focus on equilibrated estimators so that the constant in our upper
bound is fully computable and asymptotically equal to $1$, and the constant in the lower bound
is polynomial-degree-robust. Notice that similar results could be achieved with, e.g.,
a residual based estimator, at the price of introducing a (in general unknown) quasi-interpolation
constant in the upper bound, and a $p$-dependent constant in the lower bound which is due
to the use of ``bubble functions'' in the efficiency analysis.

Our setting is able to handle absorbing boundary conditions,
which appears to be new in the context of a posteriori error estimation.
We also allow for discontinuous material coefficients and (possibly) non-smooth
solutions in space. However, we require that the right-hand sides $f$ and $g$
are smooth in time and for the sake of simplicity, we neglect the approximation
of initial conditions (i.e. we assume that $u_0$ and $u_1$ are discrete functions).
We also assume for shortness that $f$ and $g$ are piecewise polynomial in space,
but we could classically alleviate this limitation by introducing the usual data
oscillation terms. These assumptions are rigorously stated in Section
\ref{section_settings}.

The key idea of this work is to measure the error in a ``damped energy norm''
\begin{equation*}
\CE_\rho^2 \eq
\int_0^{+\infty}
\left (
\|\dot u-\dot u_h\|_{\mu,\Omega}^2
+
\frac{1}{\rho} \|\dot u - \dot u_h\|_{\gamma,\GA}^2
+
\|\grad(u-u_h)\|_{\BA,\Omega}^2
\right )
e^{-2\rho t} dt,
\end{equation*}
where $\rho > 0$ is a user-defined damping parameter that can
be selected as small as desired, and $u_h$ is the (semi-discrete)
finite element approximation (notice that this norm is stronger than
the $L^\infty(L^2)$ and $L^2(H^1)$ norms of \cite{georgoulis_lakkis_makridakis_2013a}
and \cite{picasso_2010a}). We then propose an estimator $\eta(t)$
computed at each time step from the ``instantaneous'' residual
\begin{equation*}
v
\to
(\mu f(t),v)_\Omega+(\gamma g(t),v)_\GA
-
(\mu \ddot u_h(t),v)
-
(\gamma \dot u_h(t),v)_\GA
-
(\BA\grad u_h(t),\grad v)_\Omega.
\end{equation*}
The computation of $\eta(t)$ is local in both time and space, and follows the
equilibrated flux construction designed in \cite{%
braess_pillwein_schoberl_2009a,
chaumontfrelet_ern_vohralik_2021a,
destuynder_metivet_1999a,
ern_vohralik_2015a}.
We emphasize that other standard techniques, such as residual-based constructions, can
be employed to define $\eta(t)$, although the final constants in the estimates are less
explicit in this case. Then, letting
\begin{equation}
\label{eq_definition_LBA_rho}
\Lambda_\rho^2 \eq \int_0^{+\infty} \eta^2(t) e^{-2\rho t} dt,
\end{equation}
our key results read as follows. We have the upper bound
\begin{equation}
\label{eq_intro_reliability}
\CE_\rho^2 \leq
(1+4\gamma_{\rho,\omega,h}^2) \Lambda_\rho^2 + \left (\frac{\rho}{\omega}\right )^{2r} \osc_{\rho,r}^2,
\end{equation}
where $\omega > 0$ and $r \in \mathbb N$ are arbitrary, $\osc_{\rho,r}$ is a fully computable
``time-domain'' data oscillation term that involves $\rho$ weighted norms of $f^{(r)}$ and
$g^{(r)}$, and $\gamma_{\rho,\omega,h}$ is a constant such that
\begin{equation}
\label{eq_intro_gamma}
\gamma_{\rho,\omega,h} \leq \sqrt{1 + \frac{\omega}{\rho}},
\qquad
\lim_{(h/p) \to 0} \gamma_{\rho,\omega,h} = 0.
\end{equation}
The ``cut-off frequency'' $\omega$ can be freely selected, and in particular,
we can choose it large enough so that the second term in the right-hand side
of \eqref{eq_intro_reliability} is of ``higher-order'', meaning that it converges
to zero faster than the error (see Corollary \ref{corollary_reliability} below).
The estimate in \eqref{eq_intro_gamma} shows that all the constants in our upper bound
are indeed explicitly controlled, while the limit justifies that our reliability estimate is
asymptotically constant-free. On the other hand, we also show that
\begin{equation}
\label{eq_intro_efficiency}
\Lambda_\rho^2 \leq \Clb^2
\kappa_{\BA}
\left \{
\left (
\kappa_{\BA} + \frac{\rho \hs}{\cs} + \left (\frac{\omega \hs}{\cs}\right )^2
\right )
\CE_\rho^2
+
\left (\frac{\rho \hs}{\cs}\right )^2 \left (\frac{\rho}{\omega}\right )^{2r}
\osc_{\rho,r+1}^2
\right \},
\end{equation}
where $\hs/\cs \simeq \max_K h_K/\cK$ with $\cK$ the wave speed in the element $K$,
$\kappa_{\BA}$ is the ``contrast'' in the coefficient $\BA$, and $\Clb$ is a generic
constant solely depending on the shape-regularity of the mesh. All the constants appearing
in \eqref{eq_intro_efficiency} are independent of the polynomial degree $p$ of the finite
element approximation, so that our estimator is efficient and polynomial-degree-robust.
We refer the reader to Section \ref{section_main_results} for more details.

In practice, one is often interested in understanding the behaviour of the solution $u$
of \eqref{eq_wave_equation} on the time interval $(0,T)$ for some known $T > 0$. This can be
accommodated in our setting by selecting $\rho \sim 1/T$. Another important comment is that
in addition to ignoring time discretization, our setting also assumes that the simulation time
is infinite, since the definition of the estimator $\Lambda_\rho$ in \eqref{eq_definition_LBA_rho}
includes an integral over $(0,+\infty)$. This is not an important limitation in practice as long
as the simulation time is large enough since the integrand in \eqref{eq_definition_LBA_rho}
decays exponentially. In practical computations, the integral in \eqref{eq_definition_LBA_rho}
is computed by accumulating the integrand over time steps, and the simulation can be stopped when
stagnation is numerically observed. This is explained in detail in Section
\ref{section_time_discretization}.

As far as the analysis is concerned, we take an approach that is fairly different from
the aforementioned works on this topic. Indeed, our key idea is to perform a
Laplace transform and to work in the frequency domain. We emphasize that this Laplace
transform is only used as a theoretical device: it is never computed numerically. We then
rely on recent results for the Helmholtz equation to control the low-frequency content of
the error \cite{chaumontfrelet_ern_vohralik_2021a,dorfler_sauter_2013a} and introduce new
arguments to treat the high-frequency content. We refer the reader to Section
\ref{section_frequency} for an in-depth discussion.

We provide several numerical examples where linear and quadratic
Lagrange finite elements are coupled with an explicit leap-frog scheme
\cite{georgoulis_lakkis_makridakis_virtanen_2016a}.
These examples are in full agreement with our theoretical findings and suggest
that they are sharp. Besides, they seem to indicate that time-discretization
can be entirely neglected for linear elements (most likely, because they are coupled
with a second-order integration scheme here), and only moderately affects quadratic
elements when choosing a time-step close to the CFL stability limit.

The remaining of this work is organized as follows.
Section \ref{section_settings} collects all the necessary notation to properly
introduce our main results, which we state in Section \ref{section_main_results}.
In Section \ref{section_frequency}, we present the strategy we employ in our analysis
as well as some preliminary results in the frequency domain. Sections \ref{section_reliability}
and \ref{section_efficiency} respectively contain the reliability and efficiency proofs.
We present a variety of numerical examples in Section \ref{section_numeric} before giving
some concluding remarks.

\section{Settings}
\label{section_settings}

This section describes notation and recalls preliminary results.

\subsection{Domain and coefficients}

$\Omega \subset \mathbb R^d$, $d=2$ or $3$, is a polytopal Lischiptz domain.
The boundary $\partial \Omega$ is split into two relatively open, polytopal and disjoint
subsets $\GD$ and $\GA$ with Lipschitz boundaries such that
$\partial \Omega = \overline{\GD} \cup \overline{\GA}$.
The situations where $\GD = \emptyset$ or $\GA = \emptyset$ are allowed.

Our model wave equation features three coefficients, namely
$\mu: \Omega \to \mathbb R$, $\BA: \Omega \to \mathbb S^d$
and $\gamma: \GA \to \mathbb R$. For the sake of simplicity, we
assume that there exists a partition $\LP$ of $\Omega$ into non-overlaping
Lipschitz polytopal subdomains such that for all $P \in \CP$, $\mu|_P$ and $\BA|_P$ are constant
and, if $B \eq \partial P \cap \GA$ has positive surface measure, that $\gamma|_B$ is constant.

We assume that $\mu > 0$ in $\Omega$ and, for the sake of simplicity, that $\gamma > 0$
on $\GA$. We also introduce, for a.e. $\bx \in \Omega$
\begin{equation*}
a_\star(\bx) \eq \min_{\substack{\bd \in \mathbb R^d \\ |\bd| = 1}}
\BA(\bx) \bd \cdot \bd,
\qquad
a^\star(\bx) \eq \max_{\substack{\bd \in \mathbb R^d \\ |\bd| = 1}}
\BA(\bx) \bd \cdot \bd,
\end{equation*}
and require that $a_\star > 0$ in $\Omega$.

\subsection{Mesh, hat functions and vertex patches}

The domain $\Omega$ is partitioned into a computational mesh $\CT_h$ of (closed)
simplicial elements $K$. We assume that the mesh is conforming, meaning that
the intersection $K_+ \cap K_-$ of two distinct elements $K_\pm \in \CT_h$ is either a single
vertex, a full edge, or a full face of both elements. This requirement is entirely
standard and does not prevent strong mesh grading (see, e.g., \cite{boffi_brezzi_fortin_2013a}).

For an element $K \in \CT_h$, $h_K$ and $\rho_K$ denote the diameters of the smallest ball
containing $K$ and of the largest ball contained in $K$. Then, $\sr_K \eq h_K/\rho_K$
is the shape regularity parameter of $K$. If $\CT \subset \CT_h$ is a collection of
elements $\sr_\CT \eq \max_{K \in \CT} \sr_K$ is the shape-regularity parameter
of $\CT$. The notations $h_{\max} \eq \max_{K \in \CT_h} h_K$ and
$\rho_{\min} \eq \min_{K \in \CT_h}$ will also be useful.

$\CF_h$ is the set of faces of $\CT_h$. We assume that $\CT_h$ conforms with the partition
boundary, meaning that for each face $F \in \CF_h$ such that $F \subset \partial \Omega$,
either $F \subset \overline{\GA}$ or $F \subset \overline{\GD}$. The set $\CF_h^{\rm A}$
collects those faces contained in $\GA$.

We will also assume that the mesh fits the physical partition, meaning that
for all $K \in \CT_h$, there exists a $P \in \LP$ such that $K \subset \overline{P}$.
Equivalently, it means that $\mu|_K$ and $\BA|_K$ are constant for $K \in \CT_h$, and
that $\gamma|_F$ is constant for all $F \in \CF_h^{\rm A}$.

$\CV_h$ is the set of vertices of $\CT_h$. For $\ba \in \CV_h$, $\pa$ denotes the
associated hat function, that is, the only piecewise affine function on $\CT_h$ such
that $\pa(\bb) = \delta_{\ba,\bb}$ for all $\bb \in \CV_h$. We set $\oa \eq \supp \pa$.
Then, the vertex patch $\CTa \subset \CT_h$ collecting the elements $K \in \CT_h$ sharing
the vertex $\ba$ covers $\oa$.

\subsection{Key functional spaces}
\label{section_functional_spaces}

If $U \subset \Omega$ is an open set, $L^2(U)$, $H^1(U)$
and $H^2(U)$ are the usual (real-valued) Lebesgue and Soboelv spaces \cite{adams_fournier_2003a}.
$\LH^1_\GD(\Omega)$ and $\LH^2(\Omega)$ for the complex-valued counterparts
of $H^1_\GD(\Omega)$ and $H^2(\Omega)$. $L^2_0(U)$ is the subset of $L^2(U)$ of functions
with vanishing mean value, and $\BL^2(U) \eq [L^2(U)]^d$. The natural inner-products and norm
of $L^2(U)$ and $\BL^2(U)$ are denoted by $(\cdot,\cdot)_U$ and $\|{\cdot}\|_U$. We also employ the
equivalent norms $\|{\cdot}\|_{\mu,U}^2 \eq (\mu\cdot,\cdot)_U$ and
$\|{\cdot}\|_{\BA,U}^2 \eq (\BA\cdot,\cdot)_U$ on $L^2(U)$ and $\BL^2(U)$. For a relatively open
$V \subset \GA$, $L^2(V)$ is the Lebesgue space on $V$ equiped with the surface measure.
$(\cdot,\cdot)_V$ and $\|{\cdot}\|_V$ are the norm and inner-product of $L^2(V)$,
and we set $\|{\cdot}\|_{\gamma,V}^2 \eq (\gamma\cdot,\cdot)_V$.
$\BH(\ddiv,U)$ is the subspace of $\BL^2(U)$ consisting of functions with weak
$L^2(U)$ divergence \cite{girault_raviart_1986a}. If $w \subset \partial U$ is
a relatively open subset, then $H^1_w(U)$ is the subset of $H^1(U)$ of functions
with vanishing traces on $w$. Similarly, $\BH_w(\ddiv,U)$ is the subset $\BH(\ddiv,U)$
of functions with vanishing normal traces on $w$ (we refer the reader to
\cite{fernandes_gilardi_1997a} for a rigorous definition of normal traces on parts
of the boundary).

When analyzing the flux construction, we will also need local spaces associated
to vertex patches. If $\ba \in \partial \Omega$, we set $L^2_\star(\oa) \eq L^2(\oa)$,
$H^1_\star(\oa) \eq H^1_{\partial \Omega}(\oa)$ and
$\BH_0(\ddiv,\oa) \eq \BH_{\partial \oa \setminus \partial \Omega}(\ddiv,\oa)$.
On the other hand, when $\ba \in \Omega$, we let $L^2_\star(\oa) \eq L^2_0(\oa)$,
$H^1_\star(\oa) \eq H^1(\oa) \cap L^2_\star(\oa)$ and
$\BH_0(\ddiv,\oa) \eq \BH_{\partial \oa}(\ddiv,\oa)$.

\subsection{Local inequalities}

\begin{subequations}
\label{eq_local_inequalities}
For all $\ba \in \CV_h$, there exist constants $\CPa$ and $\Ctra$ such that
\begin{equation}
\|v\|_{\oa} \leq \CPa \ha \|\grad v\|_{\oa},
\qquad
\|v\|_{\partial \oa} \leq \Ctra \ha^{1/2} \|\grad v\|_{\oa}
\end{equation}
for all $v \in H^1_\star(\oa)$. Setting $\Cconta \eq 1 + \CPa \ha \|\grad \pa\|_{L^\infty(\oa)}$,
we also have
\begin{equation}
\|\grad(\pa v)\|_{\oa} \leq \Cconta \|\grad v\|_{\oa},
\end{equation}
We then set $\Cgeoa \eq \max(\Cconta,\Ctra)$, and $\Cgeo \eq 3(d+1)\max_{\ba \in \CV_h} \Cgeoa$.
\end{subequations}

\subsection{Finite element spaces}

Let $K \in \CT_h$ and $q \geq 0$. $\CP_q(K)$ is the space of polynomial function on
$K$ of degree at most $q$ and $\BCP_q(K) \eq [\CP_q(K)]^d$. We also employ the notation
$\RT_q(K) \eq \BCP_q(K) + \bx \CP_q(K)$ for the set of Raviart-Thomas polynomials
\cite{boffi_brezzi_fortin_2013a}. Similarly, if $\CT \subset \CT_h$ and $q \geq 0$, $\CP_q(\CT)$
and $\RT_q(\CT)$ are the spaces of (discontinuous) functions whose restriction to each $K \in \CT$
respectively belongs to $\CP_q(K)$ and $\RT_q(K)$.

In the remaining of this work, we fix a polynomial degree $p \geq 1$, and employ the
notations $V_h \eq \CP_p(\CT_h) \cap H^1_\GD(\Omega)$ for the space Lagrange finite element
of degree $p$, and $\LV_h$ for its complex-valued counterpart. Notice that there exists an
interpolation operator $\CI_h$ mapping $\LH^2(\Omega) \cap \LH^1_\GD(\Omega)$ into $\LV_h$ such that
\begin{equation}
\label{eq_interpolation}
\|\grad(v-\CI_h v)\|_\Omega \leq \Ci \hmax \|\grad^2 v\|_\Omega
\quad
\forall v \in \LH^2(\Omega) \cap \LH^1_\GD(\Omega),
\end{equation}
where
\begin{equation*}
\|\grad^2 v\|^2_\Omega \eq \sum_{j\ell=1}^d\left \|
\frac{\partial^2 v}{\partial \bx_j \partial \bx_\ell}
\right \|_\Omega^2
\end{equation*}
and $\Ci$ only depends on the shape-regularity parameter
of the mesh and is explicitly available \cite{arcangeli_gout_1976a,liu_kikuchi_2010a}.

\subsection{Local wave speed and contrast}

The wave speed in the direction $\bd \in \mathbb R^d$, $|\bd| = 1$
at $\bx \in \Omega$ is usually defined as
$\vartheta(\bx,\bd) \eq \sqrt{\BA(\bx) \bd \cdot \bd/\mu(\bx)}$.
Besides, if the absorbing boundary condition on $\GA$ is designed to be exact
for normally incident waves, the coefficient $\gamma$ is chosen as
$\gamma(\bx) = \vartheta(\bx,\bn)/(\BA(\bx) \bn \cdot \bn)$.
If $U \subset \Omega$ is an open set, this motivates the definition
\begin{equation}
\label{eq_definition_ca}
\vartheta_U
\eq
\min
\left (
\sqrt{\frac{\inf_U a_\star}{\sup_U \mu}},
\frac{\inf_U a_\star}{\sup_{\partial U \cap \GA} \gamma}
\right ),
\end{equation}
for the minimum wave speed in $U$, where we ignore the second
term in the minimum if $\partial U \cap \GA = \emptyset$. We also
employ the shorthand notation $\ca = \vartheta_{\oa}$ for $\ba \in \CV_h$.
The quantity $\max_{\ba \in \CV_h} \nu \ha/\ca$,
where $\nu > 0$ is a frequency will often appear in the analysis. To simplify
it, we introduce $\hs/\cs \eq \max_{\ba \in \CV_h} \ha/\ca$.
We also set $\cmin \eq \vartheta_{\Omega}$.

Our reliability estimate will also depend on the ``contrast''
in the coefficient $\BA$ that we defined by
\begin{equation}
\label{eq_definition_kappa_BA}
\kappa_{\BA,\ba}
\eq
\frac{\sup_{\oa} a^\star}{\inf_{\oa} a_\star},
\qquad
\kappa_{\BA} \eq \max_{\ba \in \CV_h} \kappa_{\BA,\ba}.
\end{equation}

\subsection{Model problem}

The domain $\Omega$ and the partition $\{\GD,\GA\}$ of its boundary being defined,
the problem is closed by specifying the right-hand sides $f$ and $g$, as well as
the initial condition $u_0$ and $u_1$.

For the right-hand sides, we require that for all $t \in \mathbb R^+$,
$f(t) \in \CP_p(\CT_h)$ and $g(t) \in \CP_p(\CF_h^{\rm A}$). The general
case can be treated by adding ``data oscillation'' terms in the estimate,
which we avoid here for the sake of shortness. We will further assume that
$f \in C^\infty(\mathbb R^+,\CP_p(\CT_h))$ and $g \in C^\infty(\mathbb R^+,\CF_h^{\rm A})$.
Our analysis do require smoothness in time, but the $C^\infty$ assumption can be lowered to
finite regularity. For the initial condition, we will for assume for the sake of simplicit
that $u_0,u_1 \in V_h$, so that the initial error at $t=0$ vanish. The general case can be
handled with additional terms in the error estimates.

The solution to the wave equation is the only function
$u \in C^\infty(\mathbb R_+,H^1_{\GD}(\Omega))$ satisfying $u(0) = u_0$,
$\dot u(0) = u_1$, and
\begin{equation}
\label{eq_definition_u}
(\mu \ddot u(t),v)_\Omega
+
(\gamma \dot u(t),v)_\GA
+
(\BA \grad u(t),\grad v)_\Omega
=
(\mu f(t),v)_\Omega
+
(\gamma g(t),v)_\GA
\end{equation}
for all $v \in H^1_\GD(\Omega)$ and $t \in \mathbb R_+$.

\subsection{Discretization}

Classically, the (semi) discrete solution is obtained by
replacing the Sobolev space $H^1_\GD(\Omega)$ in \eqref{eq_definition_u}
by its discrete counterpart $V_h$. Hence, we define the discrete solution
as the unique function $u_h \in C^\infty(\mathbb R_+,V_h)$ such that $u_h(0) = u_0$,
$\dot u_h(0) = u_1$, and
\begin{equation}
\label{eq_definition_uh}
(\mu \ddot u_h(t),v_h)_\Omega
+
(\gamma \dot u_h(t),v_h)_\GA
+
(\BA \grad u_h(t),\grad v_h)_\Omega
=
(\mu f(t),v_h)_\Omega
+
(\gamma g(t),v_h)_\GA
\end{equation}
for all $v_h \in V_h$ and $t \in \mathbb R_+$.

\subsection{Frequency-domain problem}

Similar to the results in \cite{chaumontfrelet_ern_vohralik_2021a,dorfler_sauter_2013a},
our reliability estimate includes an ``approximation factor'' $\gamma_{\rho,\omega,h}$.
In order to properly define it, we first need some notation from the frequency-domain.

For $s \in \mathbb C$ with $\rho \eq \Re s > 0$, we introduce the sesquilinear form
\begin{equation}
\label{eq_definition_bs}
b_s(v,w)
=
s^2(\mu v,w)_\Omega + s(\gamma v,w)_\GA + (\BA \grad v,\grad w)_\Omega
\quad
\forall v,w \in \LH^1_\GD(\Omega),
\end{equation}
which corresponds to the Laplace transform of the left-hand sides of
\eqref{eq_definition_u} and \eqref{eq_definition_uh}. Since $\rho > 0$, $b_s$ is coercive,
and if $\phi \in \LH^1_\GD(\Omega)$, then there exists a unique
$\LS_s^\star(\phi) \in \LH^1_\GD(\Omega)$ such that
\begin{equation}
\label{eq_definition_S}
b_s(w,\LS_s^\star(\phi))
=
|s|^2(\mu w,\phi)_\Omega + \frac{|s|^2}{\rho} (\gamma w,\phi)_\GA
\quad \forall v \in \LH^1_\GD(\Omega).
\end{equation}
With the notation
$\{\cdot\}_{\LH^1_s(\Omega)}^2
\eq
|s|^2 \|{\cdot}\|_{\mu,\Omega}^2 + (|s|^2/\rho) \|{\cdot}\|_{\gamma,\GA}^2$,
we can introduce the frequency-domain approximation factor
\begin{equation}
\label{eq_definition_hgamma}
\hgamma_{s,h}
\eq
\sup_{\substack{
\phi \in \LH^1_{\GD}(\Omega)
\\
\{\phi\}_{\LH^1_s(\Omega)} = 1
}}
\min_{v_h \in V_h} \|\grad(\LS_s^\star(\phi)-v_h)\|_{\BA,\Omega},
\end{equation}
whose definition closely follows \cite{chaumontfrelet_ern_vohralik_2021a,dorfler_sauter_2013a}.

\subsection{Approximation factor}

We are now in place to introduce the (time-domain) approximation factor.
For any cut-off frequency $\omega > 0$ and damping parameter $\rho > 0$, it is defined by
\begin{equation}
\label{eq_definition_gamma}
\gamma_{\rho,\omega,h}
\eq
\sup_{\substack{s = \rho+i\nu \\|\nu| < \omega}} \hgamma_{s,h}.
\end{equation}
The definition of $\gamma_{\rho,\omega,h}$ is a little bit intricate,
but we can easily summarize its key properties. On the one hand, it
increases when we increase $\omega$ and/or decrease $\rho$. On
the other hand, it converges to zero as $(h/p) \to 0$. Proposition
\ref{proposition_approximation_factor} states this more formally.

\begin{proposition}[Approximation factor]
\label{proposition_approximation_factor}
We have
\begin{equation}
\label{eq_approx_factor_guaranteed}
\gamma_{\rho,\omega,h}
\leq
\sqrt{1 + \frac{\omega}{\rho}}.
\end{equation}
In addition, if $\Omega$ is convex, $\BA = \BI$ and $\GA = \emptyset$, then we have
\begin{equation}
\label{eq_approx_factor}
\gamma_{\rho,\omega,h}
\leq
2\Ci
\left (
\frac{\rho \hmax}{\cmin}
+
\frac{\omega}{\rho} \frac{\omega \hmax}{\cmin}
\right ).
\end{equation}
\end{proposition}

\subsection{Local minimization problems}

Following \cite{braess_pillwein_schoberl_2009a,chaumontfrelet_ern_vohralik_2021a,ern_vohralik_2015a},
the construction of our estimator will rely on local divergence-constrained minimization problems.
The following result is paramount to establish the efficiency of the estimator, its proof can
be found in \cite{braess_pillwein_schoberl_2009a} for the 2D case, and in
\cite{ern_vohralik_2021a} for the 3D case.

\begin{proposition}[Discrete stable minimization]
\label{proposition_stable_minimization}
Let $q \in \mathbb N$, $\bchi_q \in \RT_q(\CTa)$, $d_q \in \CP_q(\CTa)$
and $b_q \in \CP_q(\CFa)$. If $\ba \notin \overline{\GD}$, assume that
\begin{equation}
\label{eq_compatibility_condition}
(d_q,1)_{\oa} = (b_q,1)_{\GA}.
\end{equation}
We have
\begin{equation}
\label{eq_stable_minimization}
\min_{\substack{
\btau_q \in \RT_q(\CTa) \cap \BH_0(\ddiv,\oa)
\\
\div \btau_q = d_q \text{ in } \oa
\\
\btau_q \cdot \bn = b_q \text{ on } \GA
}} \|\btau_q-\bchi_p\|_{\oa}
\leq
\Csta
\min_{\substack{
\btau \in \BH_0(\ddiv,\oa)
\\
\div \btau = d_q \text{ in } \oa
\\
\btau \cdot \bn = b_q \text{ on } \GA
}} \|\btau-\bchi_p\|_{\oa}
\end{equation}
where $\Csta$ only depends on the shape-regularity parameter of the patch
and in particular, does not depend on $q$.
\end{proposition}

The constant $\Cst \eq \max_{\ba \in \CV_h} \Csta$ will be useful to state
our efficiency result.


\section{Main results}
\label{section_main_results}

This section summarizes the key findings of this work.

\subsection{Equilibrated flux}

We first clarify what we mean by an equilibrated flux and propose a localized construction.

\begin{definition}[Equilibrated flux]
\begin{subequations}
\label{eq_flux}
An equilibrated flux is a function $\btau_h: \mathbb R_+ \to \BH(\ddiv,\Omega)$ such that
\begin{equation}
\div \btau_h(t) = \mu(f(t)- \ddot u_h(t)) \text{ in } \Omega
\end{equation}
and
\begin{equation}
\btau_h(t) \cdot \bn = \gamma(g(t)-\dot u_h(t)) \text{ on } \GA
\end{equation}
for all $t \in \mathbb R_+$.
\end{subequations}
For such $\btau_h$, we set
\begin{subequations}
\label{eq_estimator}
\begin{equation}
\label{eq_eta}
\eta(t) \eq \|\BA^{-1} \btau_h(t)+\grad u_h(t)\|_{\BA,\Omega},
\end{equation}
and
\begin{equation}
\label{eq_Lambda}
\Lambda_\rho^2 \eq \int_0^{+\infty} \eta(t)^2 e^{-2\rho t} dt,
\end{equation}
for $\rho > 0$.
\end{subequations}
\end{definition}

Our first contribution is a construction of such a flux that
is local in both space and time. This construction is standard,
and follows the line of
\cite{braess_pillwein_schoberl_2009a,destuynder_metivet_1999a,ern_vohralik_2015a}.
Specifically, for each $t \in \mathbb R_+$ and each vertex $\ba \in \CV_h$, we set
\begin{equation*}
\diva \eq \pa \mu (f(t)-\ddot u_h(t)) - \BA\grad \pa \cdot \grad u_h(t)
\text{ in } \oa,
\quad
\bnda \eq \pa\gamma (\dot u_h(t)-g(t))
\text{ on } \GA.
\end{equation*}
We then define local contributions by
\begin{subequations}
\label{eq_flux_construction}
\begin{equation}
\label{eq_definition_siga}
\bsig_h^{\ba}(t)
\eq
\arg \min_{\substack{
\btau_h \in \RT_{p+1}(\CTa) \cap \BH_0(\ddiv,\oa)
\\
\div \btau_h = \diva(t) \text{ in } \oa
\\
\btau_h \cdot \bn = \bnda(t) \text{ on } \GA
}}
\|\BA^{-1} \btau_h + \grad u_h(t)\|_{\BA,\oa}
\end{equation}
that we assemble as
\begin{equation}
\label{eq_definition_sig}
\bsig_h(t) \eq \sum_{\ba \in \CV_h} \bsig_h^{\ba}(t).
\end{equation}
\end{subequations}
Following the lines of \cite{%
braess_pillwein_schoberl_2009a,%
chaumontfrelet_ern_vohralik_2021a,%
destuynder_metivet_1999a,%
ern_vohralik_2015a}, we can easily show that the construction is indeed valid
and provides an equilibrated flux as per \eqref{eq_flux}. We skip the proof
here for the sake of shortness.

\begin{proposition}[Localized flux construction]
The local mixed problems in \eqref{eq_definition_siga} are well-posed, and the construction
\eqref{eq_definition_sig} provides an equilibrated flux satisfying \eqref{eq_flux}.
\end{proposition}

\subsection{Reliability}

Our next set of results concerns the reliability of the proposed estimator.
The error will be measured in the following norm
\begin{equation*}
\enorm{u-u_h}_\rho^2 \eq
\int_0^{+\infty}
\left \{
\|\dot u-\dot u_h\|_{\mu,\Omega}^2
+
\frac{1}{\rho}\|\dot u-\dot u_h\|_{\gamma,\GA}^2
+
\|\grad(u-u_h)\|_{\BA,\Omega}^2
\right \}
e^{-2\rho t}
dt
\end{equation*}
where $\rho > 0$ is an arbitrary damping parameter. Our general result reads as follows.

\begin{theorem}[Reliability]
\label{theorem_reliability}
Assume that $\btau_h$ satisfies \eqref{eq_flux} and $\Lambda_\rho$
is defined by \eqref{eq_estimator}. Then, for all $\rho,\omega > 0$
and $r \in \mathbb N$, we have
\begin{equation*}
\enorm{u-u_h}_\rho^2
\leq
(1+4\gamma_{h,\omega,\rho}^2) \Lambda_\rho^2
+
\left (\frac{\rho}{\omega}\right )^{2r}\osc_{\rho,r}^2
\end{equation*}
where
\begin{equation}
\label{eq_definition_osc}
\osc_{\rho,r}^2
=
\frac{4}{\rho^{2r}}
\left \{
\frac{1}{\rho^2}
\enorm{f^{(r)}e^{-\rho t}}_{\mu,\Omega}^2
+
\enorm{g^{(r)}e^{-\rho t}}_{\gamma,\GA}^2
\right \}.
\end{equation}
\end{theorem}

Crucially, we can select $\omega$ and $r$ large enough so that the ``oscillation'' term
converges to zero faster than the error itself. On the other hand, $\gamma_{h,\omega,\rho} \to 0$
as $(h/p) \to 0$, justifying that our upper bound is ``asymptoically constant-free''.
When the domain is convex and surrounded by a Dirichlet boundary condition,
a simpler expression can be derived.

\begin{corollary}[Simplified error estimate]
\label{corollary_reliability}
Let $\btau_h$ satisfy \eqref{eq_flux} and define $\Lambda_\rho$ by
\eqref{eq_eta}. Under the assumptions that $\Omega$ is convex,
that $\BA \equiv \BI$, and that $\GA = \emptyset$, we have
\begin{align*}
\enorm{u-u_h}_{\rho}^2
&\leq
\left (
1 + 16 \Ci^2
\left (
\left (\frac{\rho \hmax}{\cmin}\right )^{1/2}
+
\frac{\rho \hmax}{\cmin}
\right )^2
\right )
\Lambda_\rho^2
+
\left (\frac{\rho \hmax}{\cmin}\right )^{2(p+1)} \osc_{\rho,3(p+1)}^2
\end{align*}
for all $\rho > 0$.
\end{corollary}

\subsection{Efficiency}

Finally, we present our efficiency results. We start with the most general form.

\begin{theorem}[Efficiency]
\label{theorem_efficiency}
Assume that $\btau_h$ has been constructed through the construction
described in \eqref{eq_flux_construction}, and define $\Lambda_\rho$
with \eqref{eq_estimator}. The estimate
\begin{equation*}
\Lambda_\rho^2
\leq
\Clb^2
\kappa_{\BA}
\left \{
\left (
\kappa_{\BA}
+
\frac{\rho \hs}{\cs}
+
\left (\frac{\omega \hs}{\cs}\right )^2
\right )
\enorm{u-u_h}_{\rho}^2
+
\left (\frac{\rho \hs}{\cs}\right )^2 \left (\frac{\rho}{\omega}\right )^{2r} \osc_{\rho,r+1}^2
\right \}
\end{equation*}
holds true for all $\rho,\omega > 0$ and $r \in \mathbb N$, with $\Clb \eq \Cgeo \Cst$.
\end{theorem}

Similar to the reliability results, $\omega$ and $r$ can be chosen so that
the oscillation term converges to zero faster than the error. This is clearly
highlighted in Corollary \ref{corollary_efficiency}.

\begin{corollary}[Simplified efficiency estimate]
\label{corollary_efficiency}
Under the assumption of Theorem \ref{theorem_efficiency},
for all $\rho > 0$, we have
\begin{equation*}
\Lambda_\rho^2
\leq
\Clb^2
\kappa_{\BA}
\left \{
\left (
\kappa_{\BA} + 2\frac{\rho \hs}{\cs}
\right )
\enorm{u-u_h}_\rho^2
+
\left (\frac{\rho \hs}{\cs}\right )^{2(p+1)} \osc_{\rho,2(p+1)}^2
\right \}.
\end{equation*}
\end{corollary}

\section{Preliminary results in the frequency domain}
\label{section_frequency}

For the sake of simplicity, we will define $\xi_h \in C^\infty(\mathbb R_+,H^1_\GD(\Omega))$
by setting
\begin{equation}
\label{eq_definition_error}
\xi_h(t) \eq u(t)-u_h(t).
\end{equation}
Classically, a central aspect of our analysis is to view the error $\xi_h$
as a particular solution to our PDE model, the associated right-hand side being
the ``residual''. To this end, we introduce for each $t \in \mathbb R_+$ the residual functional
$\LR(t) \in \left (H^1_\GD(\Omega)\right )'$ by
\begin{equation}
\label{eq_definition_residual}
\langle \LR(t), v \rangle
\eq
(\mu(f(t)-\ddot u_h(t)),v)_\Omega
+
(\gamma (g(t)-\dot u_h(t)),v)_\GA
-
(\BA \grad u_h(t),\grad v)_\Omega
\end{equation}
for all $v \in H^1_\GD(\Omega)$, as well as its norm
\begin{equation*}
\|\LR(t)\|_{-1,\Omega}
\eq
\sup_{\substack{v \in H^1_\GD(\Omega) \\ \|\grad v\|_{\BA,\Omega} = 1}} \langle \LR, v \rangle.
\end{equation*}

Notice that then, $\xi_h(0) = \dot \xi_h(0) = 0$ due to our assumptions that $u_0,u_1 \in V_h$,
and we have
\begin{equation}
\label{eq_PDE_error}
(\mu \ddot \xi_h(t),v)_\Omega
+
(\gamma\dot \xi_h(t),v)_\GA
+
(\BA\grad \xi_h(t),\grad v)_\Omega
=
\langle \LR(t),v \rangle
\end{equation}
for all $v \in H^1_\GD(\Omega)$ and $t \in \mathbb R_+$
as can be seen by substracting \eqref{eq_definition_u} and \eqref{eq_definition_uh}.

When solving a steady problem with an inf-sup stable left-hand side,
\eqref{eq_PDE_error} can readily be employed to bound the discretizaton
error by the residual norm (see \cite[Equation (5.1)]{chaumontfrelet_ern_vohralik_2022a}
for instance). This approach is also fruitful for parabolic problems (see, e.g.,
\cite[Section 5]{ern_vohralik_2010a}), where a suitable space-time inf-sup condition
is available \cite{tantardini_veeser_2016a}. Unfortunately, to the best of the author's
knowledge, such a framework is not available for the wave equation, making the link between
the residual norm and the error harder to establish.

In this section, we develop the main idea of this work. It consists in establishing
a relation between the residual norm and the error in the frequency-domain, and
to treat distinctly the low-frequency and high-frequency content. For the low-frequency
part, the analysis follows the lines of
\cite{chaumontfrelet_ern_vohralik_2021a,dorfler_sauter_2013a},
with the main difference that here, the frequency is complex-valued, with a positive
imaginary part. We employ separate stability arguments to deal with the high-frequency
content.

We establish all our main results in terms of the residual norm $\|\LR(t)\|_{-1,\Omega}$,
that we sometimes call the ``idealized'' estimator. Indeed, we believe this form is more
general, since the residual norm can then be controlled by different types of estimators
(including the equilibrated estimator we are focusing on here).

\subsection{Laplace transform}

The key tool we employ to connect the time and frequency domains is
the Laplace transform. It is classically defined by
\begin{equation*}
\LL\{v\}(s) = \int_0^{+\infty} v(t)e^{-st} dt,
\end{equation*}
whenever the integral is properly defined. For all $\rho > 0$, we have
\begin{equation}
\label{eq_laplace_norm}
\int_0^{+\infty} |v(t)|^2 e^{-2\rho t} dt
=
\int_{\rho-i\infty}^{\rho+i\infty} |\LL\{v\}(s)|^2 ds
\end{equation}
If $v$ is sufficiently regular and with $v(0) = 0$, we have
\begin{equation}
\label{eq_laplace_diff}
\LL\{\dot v\}(s) = s\LL\{v\}(s)
\end{equation}
for all $s \in \mathbb C$ with $\Re s > 0$.
Finally, we will use the following result to estimate high-frequency
contents
\begin{equation}
\label{eq_laplace_cut}
\int_{\rho-i\infty}^{\rho+i\infty} |\LL_\rho\{v\}(s)|^2 \chi_{|s| > \mu} ds
\leq
\mu^{-2q} \int_{0}^{+\infty} |v^{(q)}(t)|^2 e^{-2\rho t}dt
\end{equation}

In the remaining of this section we fix a complex number $s \in \mathbb C$
with $\rho \eq \Re s > 0$. For the sake of shortness, we often employ
the notation $\widehat \phi \eq \LL\{\phi\}(s)$ for any function $\phi$
in the proofs.

\subsection{Frequency-domain problems}

Recalling \eqref{eq_laplace_diff}, taking the Laplace transform of \eqref{eq_definition_u}
and \eqref{eq_definition_uh}, we have
\begin{subequations}
\label{eq_freq_solutions}
\begin{equation}
b_s(\LL\{u\}(s),v) = (\mu \LL\{f\}(s),v)_\Omega + (\gamma \LL\{g\}(s),v)_\GA,
\end{equation}
and
\begin{equation}
b_s(\LL\{u_h\}(s),v_h) = (\mu \LL\{f\}(s),v_h)_\Omega + (\gamma \LL\{g\}(s),v_h)_\GA,
\end{equation}
\end{subequations}
for all $v \in \LH^1_\GD(\Omega)$ and $v_h \in \LV_h$,
where $b_s$ is the sesquilinear form defined at \eqref{eq_definition_bs}.
Similarly, noticing that the definition of $\LR$ naturally extends over
$\LH^1_\GD(\Omega)$, we can define $\LL\{\LR\}(s) \in \left (\LH_\GD(\Omega)\right )'$
by setting
\begin{equation*}
\langle \LL\{\LR\}(s),v\rangle
=
\int_0^{+\infty} \langle \LR(t),v \rangle e^{-st} dt,
\end{equation*}
for all $v \in \LH^1_\GD(\Omega)$, and we have
\begin{equation}
\label{eq_freq_residual}
b_s(\LL\{\xi_h\}(s),v) = \langle \LL\{\LR\}(s), v \rangle.
\end{equation}

\subsection{Stability}

In the following, we equip $\LH_\GD^1(\Omega)$ with norm
\begin{equation*}
\|v\|_{\LH^1_s(\Omega)}^2
=
|s|^2 \|v\|_{\mu,\Omega}^2
+
\frac{|s|^2}{\rho} \|v\|_{\gamma,\GA}^2
+
\|\grad v\|_{\BA,\Omega}^2
\quad
\forall v \in H^1_\GD(\Omega).
\end{equation*}
Straightforward arguments then show that $b_s$ is coercive in the
$\|{\cdot}\|_{\LH^1_s(\Omega)}$ norm. Specifically,
\begin{equation}
\label{eq_inf_sup}
\|v\|_{\LH^1_s(\Omega)}^2
=
\frac{1}{\rho} \Re b_s(v,sv)
\leq
\frac{|s|}{\rho} |b_s(v,v)|
\qquad
\forall v \in \LH^1_\GD(\Omega).
\end{equation}
%
%
%
As a direct consequence, we obtain a (coarse) upper bound for the frequency-domain error.

\begin{lemma}[Coarse frequency-domain upper bound]
We have
\begin{equation}
\label{eq_freq_stab}
\|\LL\{\xi_h\}(s)\|_{\LH^1_s(\Omega)}^2
\leq
4 \left \{
\frac{1}{\rho^2} \|\LL\{f\}(s)\|_{\mu,\Omega}^2 + \|\LL\{g\}(s)\|_{\gamma,\GA}^2
\right \}.
\end{equation}
\end{lemma}

\begin{proof}
On the one hand, recalling \eqref{eq_inf_sup} and \eqref{eq_freq_solutions}, we have
\begin{equation*}
\|\hu_h\|_{\LH^1_s(\Omega)}^2
\leq
\frac{|s|}{\Re s}|b_s(\hu_h,\hu_h)|
=
\frac{|s|}{\Re s}
|(\mu\hf,\hu_h)_\Omega + (\gamma \hg,\hu_h)_\Gamma|.
\end{equation*}
On the other hand, we have
\begin{align*}
|s||(\mu\hf,\hu_h)_\Omega + (\gamma \hg,\hu_h)_\Gamma|
&\leq
\|\hf\|_{\mu,\Omega}|s|\|\hu_h\|_{\mu,\Omega}
+
\rho \|\hg\|_{\gamma,\GA}\frac{|s|}{\rho}\|\hu_h\|_{\Gamma,\GA}
\\
&\leq
\left (\|\hf\|_{\mu,\Omega}^2 + \rho^2 \|\hg\|_{\gamma,\Omega}^2\right )^{1/2}
\|\hu_h\|_{\LH^1_s(\Omega)},
\end{align*}
and therefore
\begin{equation}
\label{tmp_estimate_huh}
\|\hu_h\|_{\LH^1_s(\Omega)}^2
\leq
\frac{1}{\rho^2} \|\hf\|_{\mu,\Omega}^2
+
\|\hg\|_{\gamma,\GA}^2.
\end{equation}
Similar arguments show that \eqref{tmp_estimate_huh} also holds for $u$,
and \eqref{eq_freq_stab} follows from the triangle inequality since
$\hxi_h = \hu-\hu_h$.
\end{proof}

\subsection{Approximation factor}

In order to refine the above error estimate, we will employ the approximation
factor. To simplify the discussion below, we introduce the norm
\begin{equation*}
\{\phi\}_{\LH^1_s(\Omega)}^2
\eq
|s|^2 \|\phi\|_{\mu,\Omega}^2 + \frac{|s|^2}{\rho}\|\phi\|_{\gamma,\GA}^2
\quad
\phi \in \LH^1_\GD(\Omega)
\end{equation*}
on $\LH_\GD^1(\Omega)$. Although it is not important in the forthcoming analysis, we
note that this norm is not equivalent to the usual norm, and $\LH_{\GD}^1(\Omega)$
is not an Hilbert space equipped with it. Denoting by $\Pi_h$ the orthogonal projection
onto $\LV_h$ for the $(\BA\grad\cdot,\cdot)_\Omega$ inner-product, we have
\begin{equation}
\label{eq_estimate_hgamma}
\|\grad(\LS^\star_s(\phi)-\Pi_h\LS^\star_s(\phi))\|_{\BA,\Omega}
\leq
\hgamma_{s,h}\{\phi\}_{\LH^1_s(\Omega)}
\qquad
\forall \phi \in \LH^1_s(\Omega).
\end{equation}
Explicit upper bounds for the approximation factor are available, as we next demonstrate.

\begin{lemma}[Approximation factor]
\label{lemma_freq_approx_factor}
The estimate
\begin{equation}
\label{eq_freq_approx_factor_guaranteed}
\hgamma_{s,h} \leq \frac{|s|}{\rho}
\end{equation}
holds true.
In addition, if $\Omega$ is convex, $\BA \equiv \BI$ and $\GA = \emptyset$. Then, we have
\begin{equation}
\label{eq_freq_approx_factor}
\hgamma_{s,h} \leq 2\Ci\frac{|s|}{\rho}\frac{|s|\hmax}{\cmin}.
\end{equation}
\end{lemma}

\begin{proof}
Let $\phi \in \LH^1_\GD(\Omega)$. Using \eqref{eq_inf_sup}, we have
\begin{multline*}
\|\LS_s^\star(\phi)\|_{\LH^1_s(\Omega)}^2
\leq
\frac{|s|}{\rho}|b_s(\LS_s^\star(\phi),\LS_s^\star(\phi))|
\\
=
\frac{|s|}{\rho}\{\phi\}_{\LH^1_s(\Omega)}\{\LS_s^\star\phi\}_{\LH^1_s(\Omega)}
\leq
\frac{|s|}{\rho}\{\phi\}_{\LH^1_s(\Omega)}\|\LS_s^\star(\phi)\|_{\LH^1_s(\Omega)},
\end{multline*}
so that
%
$|s|\|\LS_s^\star(\phi)\|_{\mu,\Omega} \leq \|\LS_s^\star(\phi)\|_{\LH^1_s(\Omega)} \leq (|s|/\rho) \{\phi\}_{\LH_s^1(\Omega)}$,
and \eqref{eq_freq_approx_factor_guaranteed} follows since
\begin{equation*}
\hgamma_{s,h}
\leq
\sup_{\substack{
\phi \in \LH^1_\GD(\Omega) \\ \{\phi\}_{\LH^1_s(\Omega)} = 1
}}
\|\grad \LS_s^\star(\phi)\|_{\BA,\Omega}
\leq
\sup_{\substack{
\phi \in \LH^1_\GD(\Omega) \\ \{\phi\}_{\LH^1_s(\Omega)} = 1
}}
\|\LS_s^\star(\phi)\|_{\LH^1_s(\Omega)}
\leq
\frac{|s|}{\rho}.
\end{equation*}
%
Assuming now that $\Omega$ is convex, $\BA \equiv \BI$ and $\GA = \emptyset$,
we can apply \cite[Theorem 3.2.1.2]{grisvard_1985a} and \cite[Theorem 2.2.1]{grisvard_1992a},
showing that
\begin{equation*}
\|\grad^2 \LS_s^\star(\phi)\|_\Omega
=
\|\Delta \LS_s^\star(\phi)\|_\Omega
=
\|s^2\mu\phi-s^2\mu\LS_s^\star(\phi)\|_\Omega.
\end{equation*}
Since $\BA \equiv \BI$, we have $\mu \leq \cmin^{-2}$, leading to
\begin{multline*}
\|\grad^2 \LS_s^\star(\phi)\|_\Omega
\leq
\frac{1}{\cmin}
\left (|s|^2\|\phi\|_{\mu,\Omega} + |s|^2\|\LS_s^\star(\phi)\|_{\mu,\Omega} \right )
\\
\leq
\frac{1}{\cmin}
\left (|s| + \frac{|s|^2}{\rho}\right )\{\phi\}_{\LH^1_s(\Omega)}
\leq
\frac{2|s|^2}{\cmin\rho}\{\phi\}_{\LH^1_s(\Omega)},
\end{multline*}
and \eqref{eq_freq_approx_factor} follows from
the interpolation error estimate in \eqref{eq_interpolation} since
\begin{equation*}
\min_{v_h \in \LV_h} \|\grad(\LS_s^\star(\phi)-v_h)\|_{\BA,\Omega}
\leq
\|\grad(\LS_s^\star(\phi)-\CI_h \LS^\star(\phi))\|_\Omega
\leq
\Ci \hmax
\|\grad^2 \LS_s^\star(\phi)\|_\Omega.
\end{equation*}
\end{proof}

\subsection{Reliability}

We are now ready to establish the key result of this section,
which concerns the reliability of the ``idealized estimator''
in the frequency domain.

\begin{theorem}[Frequency-domain reliability]
The estimates
\begin{equation}
\label{eq_freq_rel}
\|\LL\{\xi_h\}(s)\|_{\LH^1_s(\Omega)}
\leq
\frac{|s|}{\rho} \|\LL\{\LR\}(s)\|_{-1,\Omega}
\end{equation}
and
\begin{equation}
\label{eq_freq_rel_asym}
\|\LL\{\xi_h\}(s)\|_{\LH^1_s(\Omega)}
\leq
(1+4\hgamma_{s,h}^2)^{1/2} \|\LL\{\LR\}(s)\|_{-1,\Omega}
\end{equation}
hold true.
\end{theorem}

\begin{proof}
Selecting $v = \hxi_h$ in \eqref{eq_freq_residual}, we have
\begin{equation}
\label{tmp_bs_hxi}
|b_s(\hxi_h,\hxi_h)|
=
|\langle \widehat \LR,\hxi_h \rangle|
\leq
\|\widehat \LR\|_{-1,\Omega}\|\grad \hxi_h\|_{\BA,\Omega},
\end{equation}
and \eqref{eq_freq_rel} follows from \eqref{eq_inf_sup}.

The proof of \eqref{eq_freq_rel_asym} relies
on an ``Aubin-Nitsche trick''. Letting $\theta \eq \LS_s^\star(\hxi_h)$, picking
$v = \hxi_h$ in \eqref{eq_definition_S} and using Galerkin orthogonality,
we have
\begin{equation*}
\{\hxi_h\}_{\LH^1_s(\Omega)}^2
=
b_s(\hxi_h,\theta)
=
b_s(\hxi_h,\theta-\Pi_h\theta).
\end{equation*}
On the other hand, it follows from \eqref{eq_freq_residual}
and \eqref{eq_estimate_hgamma} that
\begin{equation*}
|b_s(\hxi_h,\theta-\Pi_h\theta)|
\leq
\|\widehat \LR\|_{-1,\Omega}
\|\grad(\theta-\Pi_h\theta)\|_{\BA,\Omega}
\leq
\hgamma_{s,h}
\|\widehat \LR\|_{-1,\Omega}
\{\hxi_h\}_{\LH^1_s(\Omega)},
\end{equation*}
so that
\begin{equation}
\label{tmp_hxi_L2}
\{\hxi_h\}_{\LH^1_s(\Omega)}
\leq
\hgamma_{s,h} \|\widehat \LR\|_{-1,\Omega}.
\end{equation}
Then, we have
\begin{equation*}
\|\grad \xi_h\|_{\BA,\Omega}^2
=
b(\hxi_h,\hxi_h) -s^2\|\hxi\|_{\mu,\Omega}^2 - s\|\hxi\|_{\gamma,\Gamma}^2
\leq
|b(\hxi_h,\hxi_h)| + |s|^2\|\hxi\|_{\mu,\Omega}^2 + |s|\|\hxi\|_{\gamma,\Gamma}^2,
\end{equation*}
and since $1 \leq |s|/\rho$, we obtain that
\begin{equation*}
\|\grad \xi_h\|_{\BA,\Omega}^2
\leq
|b(\hxi_h,\hxi_h)| + \{\hxi\}_{\LH^1_s(\Omega)}^2.
\end{equation*}
It then follows from \eqref{tmp_bs_hxi} and \eqref{tmp_hxi_L2} that
\begin{multline*}
\|\hxi_h\|_{\LH^1_s(\Omega)}^2
=
\{\hxi_h\}_{\LH^1_s(\Omega)}^2 + \|\grad \xi_h\|_{\BA,\Omega}^2
\leq
|b(\hxi_h,\hxi_h)|+2\{\hxi_h\}_{\LH^1_s(\Omega)}^2
\\
\leq
\|\widehat \LR\|_{-1,\Omega}
\|\grad \hxi\|_{\BA,\Omega}
+
2\hgamma_{s,h}^2
\|\widehat \LR\|_{-1,\Omega}^2
\leq
\frac{1}{2}\|\hxi\|_{\LH^1_s(\Omega)}^2 + \frac{1}{2} \|\widehat \LR\|_{-1,\Omega}^2
+
2\hgamma_{s,h}^2 \|\widehat \LR\|_{-1,\Omega}^2,
\end{multline*}
from which \eqref{eq_freq_rel_asym} readily follows.
\end{proof}

\section{Reliability}
\label{section_reliability}

Here, we establish our reliability results.

\subsection{The damped energy norm}

The error will be measured in a damped energy norm. For a damping
parameter $\rho > 0$, we consider the norm
\begin{equation*}
\enorm{v}_{\rho}^2
\eq
\enorm{\dot v e^{-\rho t}}_{\mu,\Omega}^2
+
\frac{1}{\rho} \enorm{\dot v e^{-\rho t}}_{\gamma,\GA}^2
+
\enorm{\grad v e^{-\rho t}}_{\BA,\Omega}^2
\end{equation*}
for all $v \in W^{1,\infty}(\mathbb R_+,H^1_\GD(\Omega))$, where we employed the notation
\begin{equation*}
\enorm{v}_{\dagger}^2 = \int_0^{+\infty} \|v(t,\cdot)\|_{\dagger}^2 dt,
\end{equation*}
for any of the ``space norms''  $\|{\cdot}\|_{\dagger}$ introduced in
Section \ref{section_functional_spaces}. It is easily seen from \eqref{eq_laplace_norm}
and \eqref{eq_laplace_diff} that actually
\begin{equation}
\label{eq_norm_identity}
\enorm{v}_\rho^2
=
\int_{\rho-i\infty}^{\rho+\infty} \|\LL\{v\}(s)\|_{\LH^1_s(\Omega)}^2.
\end{equation}

\subsection{Approximation factor}

Recall the definition of the approximation factor $\gamma_{\rho,\omega,h}$
from \eqref{eq_definition_gamma}. Then, Proposition \ref{proposition_approximation_factor}
is a direct consequence of Lemma \ref{lemma_freq_approx_factor}.

\subsection{Abstract reliability}

We now formulate our main reliability result in terms of the residual norm.
As pointed out above, this abstract formulation permits to cover other types
of estimators.

\begin{theorem}[Reliability]
\label{theorem_reliability_R}
For all $\omega > 0$, and $r \in \mathbb N$, we have
\begin{equation}
\label{eq_rel_asym}
\enorm{\xi_h}_\rho^2
\leq
(1+4\gamma_{\rho,\omega,h}^2) \enorm{\LR e^{-\rho t}}_{-1,\Omega}^2
+
\left (\frac{\rho}{\omega}\right )^{2r} \osc_{\rho,r}^2.
\end{equation}
\end{theorem}

\begin{proof}
Let $\omega > 0$. Recalling \eqref{eq_norm_identity}, we have
\begin{equation*}
\enorm{\xi_h}_\rho^2
=
\int_{\rho-i\infty}^{\rho+i\infty}
\|\LL\{\xi_h\}(s)\|_{\LH^1_s(\Omega)}^2 ds
=
I_{\rm small} + I_{\rm large}
\end{equation*}
with
\begin{equation*}
I_{\rm small}
\eq
\int_{\rho-i\infty}^{\rho+i\infty}
\|\LL\{\xi_h\}(s)\|^2_{H^1_s(\Omega)} \chi_{|s| < \omega} ds
\quad
I_{\rm large}
\eq
\int_{\rho-i\infty}^{\rho+i\infty}
\|\LL\{\xi_h\}(s)\|^2_{H^1_s(\Omega)} \chi_{|s| > \omega} ds.
\end{equation*}
On the one hand, using \eqref{eq_freq_rel_asym}, the definition
of $\gamma_{\rho,\omega,h}$ in \eqref{eq_definition_gamma} and \eqref{eq_laplace_norm},
we have
\begin{multline*}
I_{\rm small}
\leq
\int_{\rho-i\infty}^{\rho+i\infty}
(1+4\hgamma_{s,h}^2)\|\LL\{\LR\}(s)\|_{-1,\Omega}^2 \chi_{|s| < \omega} ds
\\
\leq
(1+4\gamma_{\rho,\omega,h}^2)
\int_{\rho-i\infty}^{\rho+i\infty}
\|\LL\{\LR\}(s)\|_{-1,\Omega}^2 ds
=
(1+4\gamma_{\rho,\omega,h}^2)
\enorm{\LR e^{-\rho t}}_{-1,\Omega}^2.
\end{multline*}
On the other hand, \eqref{eq_freq_stab} and \eqref{eq_laplace_cut} show that
\begin{multline*}
I_{\rm small}
\leq
4
\int_{\rho-i\infty}^{\rho+i\infty}
\left \{
\frac{1}{\rho^2} \|\LL\{f\}(s)\|_{\mu,\Omega}^2 + \|\LL\{g\}(s)\|_{\gamma,\GA}^2
\right \}
\\
\leq
\frac{4}{\omega^{2r}}
\left \{
\frac{1}{\rho^2}
\enorm{f^{(r)} e^{-\rho t}}_{\mu,\Omega}^2
+
\enorm{g^{(r)} e^{-\rho t}}_{\gamma,\GA}^2
\right \}
\end{multline*}
and \eqref{eq_rel_asym} follows recalling the definition
of $\osc_{\rho,r}$ in \eqref{eq_definition_osc}.
\end{proof}

We then give a simplified estimate assuming that holds when the $\Omega$
is convex, $\BA \equiv \BI$, and $\GA = \emptyset$. It follows by plugging
\eqref{eq_approx_factor} into \eqref{eq_rel_asym}, and then
selecting $r=3(p+1)$ and $\omega > 0$ such that
$(\omega/\rho)^3 = \omega \hmax/\cmin$.


\begin{corollary}[Simplified error estimate]
\label{corollary_reliability_R}
Assume that $\Omega$ is convex, that $\BA \equiv \BI$ and that $\GA = \emptyset$.
Then, for all $\rho > 0$, we have
\begin{multline}
\label{eq_simplified_error_estimate_R}
\enorm{\xi_h}_\rho^2
\leq
\left (
1 + 16\Ci^2
\left (
\left (
\frac{\rho \hmax}{\cmin}
\right )^{1/2}
+
\frac{\rho \hmax}{\cmin}
\right )^2
\right )
\enorm{\LR e^{-\rho t}}_{-1,\Omega}^2
\\
+
\left (
\frac{\rho \hmax}{\cmin}
\right )^{2(p+1)}
\osc_{\rho,3p+1}^2.
\end{multline}
\end{corollary}

\subsection{Application to the equilibrated estimator}

We conclude our reliability analysis by showing that the residual
norm can be controlled by the equilibrated estimator using
a Prager-Synge type argument \cite{prager_synge_1947a}. We omit the
proof as it is classical (see, e.g.,
\cite[Proposition 4.1]{chaumontfrelet_ern_vohralik_2021a}).

\begin{proposition}[Control of the residual]
\label{proposition_residual_control}
If $\btau_h$ satisfies \eqref{eq_flux} and $\eta$ and $\Lambda_\rho$
are defined by \eqref{eq_eta}, we have
\begin{equation*}
\|\LR(t)\|_{-1,\Omega} \leq \eta(t) \;\; \forall t \in \mathbb R_+,
\qquad
\enorm{\LR e^{-\rho t}}_{-1,\Omega} \leq \Lambda_\rho \;\; \forall \rho > 0.
\end{equation*}
\end{proposition}

At that point, Theorem \ref{theorem_reliability} and Corollary
\ref{corollary_reliability} follow from Theorem
\ref{theorem_reliability_R}, Corollary \ref{corollary_reliability_R}
and Proposition \ref{proposition_residual_control}.

\section{Efficiency}
\label{section_efficiency}

In this section, we establish efficiency properties of the proposed
estimator. The approach is in part similar to the analysis presented
for the time-harmonic equations in \cite{chaumontfrelet_ern_vohralik_2021a,dorfler_sauter_2013a},
but additional arguments are required to treat the second-time derivative (as opposed to a
multiplication by $-\omega^2$).

\subsection{Localized norm}

An important ingredient of the forthcoming analysis will
be the following localized residual norm
\begin{equation}
\label{eq_definition_localized_norm}
\|\LR(t)\|_{-1,\ba}
\eq
\sup_{\substack{v \in H^1_\star(\oa) \\ \|\grad v\|_{\BA,\Omega} = 1}}
\langle \LR(t), \pa v\rangle
\end{equation}
associated with each vertex $\ba \in \CV_h$. We start with a standard dual
characterization showing that this norm is actually the minimum of a continuous
version of the minimization problem defining the local flux contribution $\bsig_h^{\ba}$ in
\eqref{eq_definition_siga}. The proof is standard (see, e.g.,
\cite[Lemma 4.3]{chaumontfrelet_ern_vohralik_2021a}) and omitted here for shortness.

\begin{lemma}[Dual characterization]
We have
\begin{equation*}
\|\LR(t)\|_{-1,\ba}
=
\min_{\substack{
\bv \in \BH_\Ga(\ddiv,\oa)
\\
\div \bv = \diva \text{ in } \oa
\\
\bv \cdot \bn = \bnda \text{ on } \GA
}}
\|\BA^{-1} \bv+\pa\grad u_h(t)\|_{\BA,\oa}
\end{equation*}
for all vertices $\ba \in \CV_h$ and all $t \in \mathbb R_+$.
\end{lemma}

To establish a link between the localized and global norms of the residual,
we observe that because the $\pa$ form a partition of unity, we have
\begin{equation*}
\langle \LR(t), v \rangle 
=
\sum_{\ba \in \CV_h}
\langle \LR(t), \pa v \rangle 
\leq
\left (
\sum_{\ba \in \CV_h} \|\LR(t)\|_{-1,\ba}^2
\right )^{1/2}
\left (
\sum_{\ba \in \CV_h} \|\grad v\|_{\BA,\oa}^2
\right )^{1/2}.
\end{equation*}
Since each simplex $K$ has $d+1$ vertices, the we obtain the following upper bound:
\begin{equation}
\label{eq_glob_loc_R}
\|\LR(t)\|_{-1,\Omega}^2
\leq
(d+1) \sum_{\ba \in \CV_h} \|\LR(t)\|_{-1,\ba}^2
\quad
\forall t \in \mathbb R_+.
\end{equation}

\subsection{Abstract efficiency}

We now derive efficiency results for the idealized estimator.
We first show that point-wise in time and patch-wise in space,
the idealized estimator is a lower bound for a measure of the error
that in addition to the desired norm, also includes the second time-derivative.

\begin{lemma}[Local efficiency]
\label{lemma_local_efficiency_R}
For all vertices $\ba \in \CV_h$ and $t \in \mathbb R_+$, the estimate
\begin{multline}
\label{eq_local_efficiency_R}
\|\LR(t)\|_{-1,\ba}
\leq
\\
\Cgeoa
\left (
\frac{\ha}{\ca} \|\ddot \xi_h(t)\|_{\mu,\oa}
+
\left (\frac{\rho\ha}{\ca}\right )^{1/2} \frac{1}{\rho^{1/2}}
\|\dot \xi_h(t)\|_{\gamma,\oa \cap \GA}
+
\kappa_{\BA}^{1/2} \|\grad \xi_h(t)\|_{\BA,\oa}
\right )
\end{multline}
holds true.
\end{lemma}

\begin{proof}
Let $t \in \mathbb R_+$ and $v \in H^1_\GD(\Omega)$. Recalling
\eqref{eq_PDE_error}, we have
\begin{align*}
|\langle \LR(t), \pa v\rangle|
&=
|
(\mu \ddot \xi_h(t),\pa v)_\Omega
+
(\gamma \dot \xi_h(t),\pa v)_\GA
+
(\BA\grad \xi_h(t),\grad (\pa v))_\Omega
|
\\
&\leq
\|\ddot \xi_h(t)\|_{\mu,\Omega}\|v\|_{\mu,\Omega}
+
\|\dot \xi_h(t)\|_{\gamma,\GA}\|v\|_{\gamma,\GA}
+
\|\grad \xi_h(t)\|_{\BA,\Omega}\|\grad(\pa v)\|_{\BA,\Omega}.
\end{align*}
Then, \eqref{eq_local_efficiency_R} follows by applying
\eqref{eq_local_inequalities}, recalling
the definitions of $\ca$ and $\kappa_{\BA}$ in \eqref{eq_definition_ca}
and \eqref{eq_definition_kappa_BA} as well as the definition of the
localized residual norm in \eqref{eq_definition_localized_norm}.
\end{proof}

Next, we show that globally in time and space, the second time-derivative can be
removed, at the price of introducing a data oscillation term.

\begin{lemma}[Second time-derivative]
\label{lemma_second_time_derivative}
For all $\omega > 0$ and $r \in \mathbb N$, we have
\begin{equation}
\label{eq_second_time_derivative}
\enorm{\ddot \xi_h e^{-\rho t}}_{\mu,\Omega}^2
\leq
\omega^2 \enorm{\dot \xi_h e^{-\rho t}}_{\mu,\Omega}^2
+
\rho^2 \left (\frac{\rho}{\omega}\right )^{2r} \osc_{\rho,r+1}^2.
\end{equation}
\end{lemma}

\begin{proof}
Standard identity \eqref{eq_laplace_norm} of the Laplace transform shows that
\begin{align*}
\enorm{\ddot \xi_h e^{-\rho t}}_{\mu,\Omega}^2
=
\int_{\rho-i\infty}^{\rho+i\infty}
\|\LL\{\ddot \xi_h\}(s)\|_{\mu,\Omega}^2 ds.
\end{align*}
We split the integral in the right-hand side in two. On the one hand,
using \eqref{eq_laplace_diff} and \eqref{eq_laplace_norm} again, we have
\begin{multline*}
\int_{\rho-i\infty}^{\rho+i\infty}
\|\LL\{\ddot \xi_h\}(s)\|_{\mu,\Omega}^2 \chi_{|s| < \omega} ds
=
\int_{\rho-i\infty}^{\rho+i\infty}
|s|^2\|\LL\{\dot \xi_h\}(s)\|_{\mu,\Omega}^2 \chi_{|s| < \omega} ds.
\\
\leq
\omega^2
\int_{\rho-i\infty}^{\rho+i\infty}
|\LL\{\dot \xi_h\}(s)\|_{\mu,\Omega}^2 ds
=
\omega^2 \enorm{\dot \xi_h e^{-\rho t}}_{\mu,\Omega}^2.
\end{multline*}
On the other hand, using \eqref{eq_laplace_diff} and \eqref{eq_freq_stab}
\begin{multline*}
\int_{\rho-i\infty}^{\rho+i\infty}
\|\LL\{\ddot \xi_h\}(s)\|_{\mu,\Omega}^2 \chi_{|s| > \omega} ds
\leq
\int_{\rho-i\infty}^{\rho+i\infty}
|s|^4 \|\LL\{\xi_h\}(s)\|_{\mu,\Omega}^2 \chi_{|s| > \omega} ds
\\
\leq
\frac{4}{\rho^2}
\int_{\rho-i\infty}^{\rho+i\infty}
|s|^2 |\LL\{f\}(s)\|_{\mu,\Omega}^2 \chi_{|s| > \omega} ds
+
4
\int_{\rho-i\infty}^{\rho+i\infty}
|s|^2 |\LL\{g\}(s)\|_{\gamma,\GA}^2 \chi_{|s| > \omega} ds.
\end{multline*}
Recalling the definition of $\osc_{\rho,r}$ in \eqref{eq_definition_osc},
we conclude the proof with \eqref{eq_laplace_diff} and \eqref{eq_laplace_cut},
since
\begin{align*}
\int_{\rho-i\infty}^{\rho+i\infty}
|s|^2 |\LL\{f\}(s)\|_{\mu,\Omega}^2 \chi_{|s| > \omega} ds
=
\int_{\rho-i\infty}^{\rho+i\infty}
|\LL\{\dot f\}(s)\|_{\mu,\Omega}^2 \chi_{|s| > \omega} ds
\leq
\omega^{-2r}
\enorm{f^{(1+r)} e^{-\rho t}}_{\mu,\Omega}^2,
\end{align*}
and a similar estimate holds for $g$, for all $r \in \mathbb N$.
\end{proof}

\begin{remark}[Localized treatment of the second time-derivative]
It is possible to obtain a partially localized version of \eqref{eq_second_time_derivative},
with the term $\enorm{\ddot \xi_h e^{-\rho t}}_{\mu,\Omega}$ and
$\enorm{\dot \xi_h e^{-\rho t}}_{\mu,\Omega}$  respectively replaced by
$\enorm{\ddot \xi_h e^{-\rho t}}_{\mu,\oa}$ and $\enorm{\dot \xi_h e^{-\rho t}}_{\mu,\oa}$.
However, the author does not see a way to localize the data oscillation term.
\end{remark}

Combining Lemmas \ref{lemma_local_efficiency_R} and \ref{lemma_second_time_derivative},
we arrive at our key efficiency result.

\begin{theorem}[Global efficiency]
\label{theorem_efficiency_R}
The estimate
\begin{equation}
\label{eq_global_efficiency_R}
\enorm{\LR e^{-\rho t}}_{-1,\Omega}^2
\leq
\Cgeo^2
\left \{
\left (
\kappa_{\BA}
+
\frac{\rho \hs}{\cs}
+
\left (\frac{\omega \hs}{\cs}\right )^2
\right )
\enorm{\xi}_{\rho}^2
+
\left (\frac{\rho \hs}{\cs}\right )^2
\left (\frac{\rho}{\omega}\right )^{2r} \osc_{\rho,r+1}^2
\right \}
\end{equation}
holds true for all $\omega > 0$ and $r \in \mathbb N$.
\end{theorem}

\begin{proof}
Recall the esimates in \eqref{eq_local_inequalities}. By summing the local estimates in
\eqref{eq_local_efficiency_R} over $\ba \in \CV_h$, and using \eqref{eq_glob_loc_R}, we obtain
\begin{align*}
\|\LR(t)\|_{-1,\Omega}^2
&\leq
(d+1) \max_{\ba \in \CV_h} \Cgeoa^2
\\
&\times
\sum_{\ba \in \CV_h}
\left (
\frac{\ha}{\ca} \|\ddot \xi_h(t)\|_{\mu,\oa}
+
\left (\frac{\rho\ha}{\ca}\right )^{1/2} \frac{1}{\rho^{1/2}}
\|\dot \xi_h(t)\|_{\gamma,\oa \cap \GA}
+
\kappa_{\BA,\ba} \|\grad \xi_h(t)\|_{\BA,\oa}
\right )^2
\\
&\leq
\Cgeo^2
\left (
\left (\frac{\hs}{\cs}\right )^2 \|\ddot \xi_h(t)\|_{\mu,\Omega}^2
+
\frac{\rho\hs}{\cs} \frac{1}{\rho}
\|\dot \xi_h(t)\|_{\gamma,\GA}^2
+
\kappa_{\BA}^2 \|\grad \xi_h(t)\|_{\BA,\Omega}^2
\right ),
\end{align*}
and therefore
\begin{equation*}
\enorm{\LR e^{-\rho t}}_{-1,\Omega}^2
\leq
\Cgeo^2
\left (
\left (\frac{\hs}{\cs}\right )^2 \|\ddot \xi_h e^{-\rho t}\|_{\mu,\Omega}^2
+
\frac{\rho\hs}{\cs} \frac{1}{\rho}
\|\dot \xi_h e^{-\rho t}\|_{\gamma,\GA}^2
+
\kappa_{\BA}^2 \|\grad \xi_h e^{-\rho t}\|_{\BA,\Omega}^2
\right ).
\end{equation*}
We then apply \eqref{eq_second_time_derivative}, showing that
\begin{equation*}
\left (\frac{\hs}{\cs}\right )^2 \|\ddot \xi_h e^{-\rho t}\|_{\mu,\Omega}^2
\leq
\left (
\frac{\omega \hs}{\cs}
\right )^2
\|\dot \xi_h e^{-\rho t}\|_{\mu,\Omega}^2
+
\left (
\frac{\rho \hs}{\cs}
\right )^2
\left (\frac{\rho}{\omega}\right )^{2r}
\osc_{\rho,r+1}^2,
\end{equation*}
and leading to \eqref{eq_global_efficiency_R}.
\end{proof}

For the reader's convenience, we also provide a simlified version
of our efficiency estimate. It is obtained from
\eqref{eq_global_efficiency_R} by selecting
$\omega > 0$ such that $(\omega\hs/\cs)^2 = \rho\hs/\cs$
and $r=2p$.

\begin{corollary}[Simplified efficiency estimate]
\label{eq_corollary_efficiency_R}
We have
\begin{equation}
\label{eq_simplified_efficiency_R}
\enorm{\LR e^{-\rho t}}_{-1,\Omega}^2
\leq
\Cgeo^2
\left \{
\left (
\kappa_{\BA}
+
2\frac{\rho \hs}{\cs}
\right )
\enorm{\xi}_{\rho}^2
+
\left (\frac{\rho}{\omega}\right )^{2(p+1)} \osc_{\rho,2p+1}^2
\right \}.
\end{equation}
\end{corollary}


\subsection{Application to the equilibrated estimator}

Theorem \ref{theorem_reliability} and Corollary \ref{corollary_reliability}
simply follow from Proposition \ref{proposition_stable_minimization},
Theorem \ref{theorem_efficiency_R} and Corollary \ref{corollary_reliability_R}.

\section{Numerical examples}
\label{section_numeric}

This section presents a set of 2D numerical examples. We use elements of degree $p=1$
or $2$, coupled with an explicit leap-frog scheme for time integration
\cite{georgoulis_lakkis_makridakis_virtanen_2016a}. The meshes are generated with
the {\tt mmg} software \cite{dobrzynski_2012a}.

\subsection{Time discretization}
\label{section_time_discretization}

Considering a time step $\dt > 0$ we set $t_n = n\dt$, and build a sequence
of approximation $(u_h^n)_{n \in \mathbb N}$, where $u_h^n$ is meant to
approximate $u(t_n)$. We employ a leap-frog scheme to mimick the time-derivative
\cite{georgoulis_lakkis_makridakis_virtanen_2016a}. Specifically, we introduce
\begin{equation*}
D^2_{\dt} u_h^n = \frac{u_h^{n+1}-2 u_h^n+u_h^{n-1}}{\dt^2},
\qquad
D  _{\dt} u_h^n = \frac{u_h^{n+1}-u_h^{n-1}}{2\dt}.
\end{equation*}
The sequence $(u_h^n)_{n \in \mathbb N}$ is then define by setting
$u_h^0 = 0$, $u_h^1 = 0$, and then iteratively by requiring that
\begin{equation}
\label{eq_fully_discrete}
(\mu D^2_{\dt} u_h^n,v_h)_\Omega
+
(\gamma D_{\dt} u_h^n,v_h)_{\GA}
+
(\BA\grad u_h^n,\grad v_h)_\Omega
\\
=
(\mu f(t_n),v_h)_\Omega + (\gamma g(t_n),v_h)_{\GA}
\end{equation}
for all $v_h \in V_h$. Classically, the choice of $\dt$ is restricted by a so-called
``CFL condition'', hence we set
\begin{equation}
\label{eq_cfl}
\dt = \alpha \min_{K \in \CT_h} \frac{\rho_K}{\vartheta_K}
\end{equation}
where $\alpha > 0$ is selected sufficiently small to obtain a stable discretization
(recall that $\rho_K$ is inscribed diameter of the element $K$ and that
$\vartheta_K$ is the wave speed inside $K$). For the meshes we employ,
we empirically find that the values $\alpha = 1.5$ when $p=1$
and $\alpha=0.6$ when $p=2$ are close to the CFL limit.

This leads to an explicit time-integration scheme where only the matrix representation
of the form
\begin{equation*}
V_h \ni w_h,v_h \to \mathbb R \to
\frac{1}{\dt^2} (\mu w_h,v_h)_\Omega + \frac{1}{2\dt} (\gamma w_h,v_h)_\GA \in \mathbb R
\end{equation*}
needs to be inverted. For the sake of simplicity, we employ the {\tt mumps}
package \cite{amestoy_duff_lexcellent_2000a} to obtain the Choleski factorization
of the above matrix in our computations. In practice however, it is possible to
employ mass-lumping to avoid factorizing the matrix \cite{cohen_joly_roberts_tordjman_2001a}.

In practice, we can only compute a finite number of iterates, meaning that the computations
stop after $N \in \mathbb N$ steps, with associated time $T \eq t_N$. Our estimates require
an infinite simulation time, but this can be easily circumvented by monitoring
\begin{equation*}
\int_0^T \eta(t)^2 e^{-2\rho t} dt
\end{equation*}
as $T$ increases and stopping the simulation when it stagnates.

For each time step $n \in \{0,\dots,N\}$, and each vertex $\ba \in \CV_h$,
we employ the time-discrete version of \eqref{eq_definition_siga} to define $\bsig_h^{\ba,n}$
\begin{equation}
\label{eq_bad_minimization}
\bsig_h^{\ba,n}
\eq
\arg \min_{\substack{
\btau_h \in \RT_{p+1}(\CTa) \cap \BH_0(\ddiv,\oa)
\\
\div \btau_h = \mu\pa(f(t_n)-D^2_{\dt} u_h^n)-\BA\grad\pa\cdot\grad u_h^n \text{ in } \oa
\\
\btau_h \cdot \bn = \gamma\pa(g(t_n)-D_{\dt} u_h^n) \text{ on } \GA
}}
\|\BA^{-1}\btau_h+\grad u_h^n\|_{\BA,\oa},
\end{equation}
and we set
\begin{equation*}
\eta_K^n \eq \|\BA^{-1} \bsig_h^n+\grad u_h^n\|_{\BA,K},
\qquad
\eta^n \eq \left (\sum_{K \in \CT_h} (\eta_K^n)^2\right )^{1/2},
\end{equation*}
as well as
\begin{equation}
\label{eq_Lambda_partial}
\Lambda_\rho^2 \eq \frac{1}{2\dt^2} \sum_{n=0}^{N}
\left (
(\eta^{n+1})^2 e^{-2\rho t_{n+1}} + (\eta^n)^2 e^{-2\rho t_n}
\right ).
\end{equation}
Notice that the compatibility condition in \eqref{eq_bad_minimization} is satisfied
due to \eqref{eq_fully_discrete}.

\subsection{Standing wave}

We consider the unit square $\Omega = (0,1)^2$ surrounded by a Dirichlet
boundary condition with $\GD = \partial \Omega$. We consider the solution
\begin{equation*}
u(t,\bx) = \chi(t) \sin(\sqrt{2}\pi t)\sin(\pi \bx_1)\sin(\pi\bx_2),
\end{equation*}
where $\chi$ is a cutoff function defined as follows. For $t \leq 0$,
$\chi(t) = 0$ and $\chi(t) = 1$ when $t \geq 1$. In the interval $[0,1]$
$\chi$ is defined as the unique $\CP_5$ polynomial that enables $C^2$
junctions. Observe that for $t \geq 1$, $u$ solves the wave equation with
a vanishing right-hand side $f$. We perform the simulation on the interval
$(0,T)$ with $T \eq 10$.

The goal of this example is to emphasize the role of the damping parameter
$\rho > 0$. First, Figure \ref{figure_standing_history} presents the evolution
of the instantaneous error (recall that $\GA = \emptyset$ here)
\begin{equation*}
\CE^2(t) \eq \|\dot \xi(t)\|_{\mu,\Omega}^2 + \|\grad \xi(t)\|_{\BA,\Omega}^2,
\end{equation*}
and the ``instantaneous'' estimator $\eta(t)$, as well as the cumulated error
\begin{equation*}
\CC_\rho^2(t) \eq \int_0^t \CE^2(\tau) e^{-2\rho \tau} d\tau,
\end{equation*}
and the cumulated estimator $\Lambda_\rho(t)$ obtained by summing \eqref{eq_Lambda_partial}
up to the current time step. On Figure \ref{figure_standing_history_inst},
we can see that $\CE(t)$ globally increases with $t$, whereas $\eta(t)$ remains globally
constant. The reason is that, similar to what happens in the frequency domain
\cite{chaumontfrelet_ern_vohralik_2021a}, the estimator is insensitive to the dispersion
error that accumulates over time. The effect is counter balanced by the parameter $\rho$
in the damped energy norm. As shown on Figure \ref{figure_standing_history_cumu_1.0000},
\ref{figure_standing_history_cumu_0.5000} and \ref{figure_standing_history_cumu_0.2500},
for a fixed mesh, the estimator tends to underestimate the error as $\rho$ is decreased.

Figure \ref{figure_standing_convergence} presents the behaviour of the estimator
$\Lambda_\rho$ and the error as the mesh is refined. As claimed, our estimate is
asymptotically constant-free: the two curves becomes indintinguishable as $h \to 0$.
The effect is analyzed in more depth on Figure \ref{figure_standing_efficiency},
where the effictivity index ${\rm eff}_\rho \eq \Lambda_\rho/\|u-u_h\|_{\rho}$
is plotted against the mesh size for different values of $\rho$, $\alpha$ and $p$.
We can see on Figures \ref{figure_standing_efficiency_P0} and \ref{figure_standing_efficiency_P1}
that for $p=1$ and $p=2$ (with an over-refined time step), we observe exactly the behaviour
predicted by our analysis: the estimate is asymtotically constant-free, and the asymptotic
regime is achieved faster when $\rho$ and $p$ are larger. On
Figure \ref{figure_standing_efficiency_P1_cfl}, we display the same results for $p=2$
and a ``natural'' time step close to the CFL stability limit. We can see in this case
that the time-discretization error have a small contribution that is not capture by
the estimator.

\input{figures/standing/history}

\input{figures/standing/convergence}

\input{figures/standing/efficiency}

\subsection{Reflections by Dirichlet boundaries}

In this example, the domain is the unit square $\Omega \eq (0,1)^2$.
On the ``bottom-left'' boundary of $\Omega$, we impose a Dirichlet
boundary condition, whereas on the ``top-right'' part, we place an
absorbing boundary condition.

Let us first consider the travelling wave
\begin{equation*}
v_{\sigma,\bd}(t,\bx) = p_{\sigma}((t-t_0)-\bd \cdot \bx),
\end{equation*}
where $t_0 = 4$, $\bd = (\cos\theta,\sin\theta)$ with
$\theta = 11\pi/8$, and $p_\sigma$ is the profile
\begin{equation*}
p_\sigma(\tau) = \tau e^{-(\tau/\sigma)^2} \qquad \tau \in \mathbb R.
\end{equation*}

We then construct our analytical solution by the principle of images,
that is by ``mirroring'' $v_{\sigma,\bd}$ along the $\{\bx_1 = 0\}$
and $\{\bx_2 = 0\}$ hyperplanes. Specifically, we set
\begin{equation*}
u(t,\bx)
=
v_{\sigma,\bd}(t,\bx_1,\bx_2)
-
v_{\sigma,\bx}(t,-\bx_1,\bx_2)
-
v_{\sigma,\bx}(t,\bx_1,-\bx_2)
+
v_{\sigma,\bd}(t,-\bx_1,-\bx_2).
\end{equation*}
It is easily seen that $u$ satisfies the volume equation
with $f = 0$ in $\Omega$ and the Dirichlet condition on $\GD$.
We then set $g = \dot u + \grad u \cdot \bn$ on $\GA$. The
initial condition $u(0,\bx)$ and $\dot u(0,\bx)$ are not exactly
zero, but are sufficiently small that setting $u^0 = u^1 = 0$ in
our computations generates a level of error less than numerical
discretization. The simulation time is $T \eq 10$, and we set
$\rho \eq 1/T$.

Here, our goal is to highlight the sensitivity of the estimator
to the oscillations present in the right-hand side, that are here
described by the parameter $\sigma$. Figure \ref{figure_reflection_history}
represents the evoluation of the instantaneous and cumulated errors
\begin{equation*}
\CE_\rho^2(t)
\eq
\|\dot \xi_h(t)\|_{\mu,\Omega}^2
+
\frac{1}{\rho} \|\dot u(t)\|_{\gamma,\GA}^2
+
\|\grad \xi_h(t)\|_{\BA,\Omega}^2,
\qquad
\CC_\rho^2(t)
\eq
\int_0^t \CE_\rho^2(\tau) e^{-2\rho \tau} d\tau,
\end{equation*}
and estimators, whereas Figure \ref{figure_reflection_convergence} presents the behaviour
of the estimator for $\sigma = 0.5$.

Figure \ref{figure_reflection_efficiency} displays the effictivity index for different mesh
sizes and values of $\sigma$ and $p$, $\alpha$. As predicted by our analysis, if the
time-step is sufficiently fine (Figures \ref{figure_reflection_efficiency_P0}
and \ref{figure_reflection_efficiency_P1}), the estimate becomes asymptotically constant-free
as the mesh is refined for any value of $\sigma$, and the asymptotic regime is
achieved faster if $p$ is large or if $\sigma$ is small (i.e. the data oscillate less).
When considering $p=2$ with a large time-step (Figure \ref{figure_reflection_efficiency_P1_cfl}),
the contribution of the time-discretization, not included in our analsysis, is clearly
visible.

On Figure \ref{figure_reflection_pictures}, we represent the true solution, the instantaneous
error $\CE_\rho(t)$, and $\eta(t)$ for different times $t$. We see that the estimators correctly
locates the error at all times, although there is an underetimation close the abosrbing boundary.
This is due to the spatial oscilation term that we have not included in the estimator
(the right-hand side is not piecewise polynomial here). We also see that as the time increases,
even if the error is correclty located, its magnitude becomes underestimated, which is in
agreement with previous numerical observations and theoretical predictions. Again, this is
because the estimator does not capture the (increasing) dispersion error.

\input{figures/reflection/history}

\input{figures/reflection/convergence}

\input{figures/reflection/efficiency}

\input{figures/reflection/pictures}

\subsection{Reflections by a penetrable obstacle}

We consider an example similar to the previous one, where an incident wave
is reflected. As before, we set $\bd = (\cos \theta,\sin \theta)$ with $\theta = 11\pi/8$,
and consider the incident field $u_{\rm inc} \eq v_{\sigma,\bd}$ for different values of
$\sigma$. We then define $g \eq \dot u_{\rm inc} + \grad u_{\rm inc} \cdot \bn$
on $\GA$. The coefficients are define by
\begin{equation*}
\mu \eq \left |
\begin{array}{ll}
\mu_D &\text{ in } D
\\
1 &\text{ outside}
\end{array}
\right .
\qquad
\BA \eq \left |
\begin{array}{ll}
\BA_D &\text{ in } D
\\
\BI &\text{ outside },
\end{array}
\right .
\end{equation*}
with $\mu_D = 2$ and $\BA_D = (1/2)\BI$, where $D$ is the polygon defined
by joining the vertices $(0,-0.5)$, $(0.5,0.5)$, $(0,0)$, $(0,0.5)$ and $(-0.5,0)$
(see Figure \ref{figure_penetrable_pictures}). As before, $T \eq 10$ and $\rho \eq 1/T$.

Here, the analytical solution is not available. As a result we limit our investigation
to $p=1$ with $\alpha=1.5$. For a reference solution, we employ our numerical scheme with
$p=2$ on the same mesh with a time-step $\dt$ divided by 3. As can be seen on Figures
\ref{figure_penetrable_efficiency} and \ref{figure_penetrable_pictures}, the behaviour of the
estimator is similar to the previous experiments, and in complete agreement with our analysis.

\input{figures/penetrable/efficiency}

\input{figures/penetrable/pictures}

\section{Conclusion}
\label{section_conclusion}

We present a construction of an equilibrated estimator for the scalar wave equation.
Our construction avoids elliptic reconstructions, and is similar to \cite{bernardi_suli_2005a}.
The key novelty of our work is to employ a damped energy norm to measure the error together
with a careful reliability and efficiency analysis providing a guaranteed and asymptotically
constant-free upper bound as well as a polynomial-degree-robust lower bound. Numerical examples
highlight the theory and suggest that it is sharp.
This work is currently limited to the semi-discretization in space, and future work will be
guided towards taking into account time-discretization, for instance following
\cite{georgoulis_lakkis_makridakis_virtanen_2016a,gorynina_lozinski_picasso_2019a}.
Besides, it would be interesting to provide equilibrated flux constructions that can
operate in the presence of mass-lumping \cite{cohen_joly_roberts_tordjman_2001a},
or to consider discontinuous Galerkin schemes
\cite{ern_vohralik_2015a,grote_schneebeli_schotzau_2006a}.

\bibliographystyle{amsplain}
\bibliography{bibliography,specific}

\end{document}

%% file: general_commands.tex
\usepackage{amssymb}
\usepackage{mathrsfs}

\renewcommand{\Re}{\operatorname{Re}}

\newcommand{\eq}{:=}

\newcommand{\grad}{\boldsymbol \nabla}
\renewcommand{\div}{\grad \cdot}

\newcommand{\ddiv}{\operatorname{div}}

\newcommand{\BA}{\boldsymbol A}

\newcommand{\BH}{\boldsymbol H}
\newcommand{\BI}{\boldsymbol I}

\newcommand{\BL}{\boldsymbol L}

\newcommand{\ba}{\boldsymbol a}
\newcommand{\bb}{\boldsymbol b}

\newcommand{\bd}{\boldsymbol d}

\newcommand{\bn}{\boldsymbol n}

\newcommand{\bv}{\boldsymbol v}

\newcommand{\bx}{\boldsymbol x}

\newcommand{\CC}{\mathcal C}

\newcommand{\CE}{\mathcal E}
\newcommand{\CF}{\mathcal F}

\newcommand{\CI}{\mathcal I}

\newcommand{\CP}{\mathcal P}

\newcommand{\CT}{\mathcal T}

\newcommand{\CV}{\mathcal V}

\newcommand{\LH}{\mathscr H}

\newcommand{\LL}{\mathscr L}

\newcommand{\LP}{\mathscr P}

\newcommand{\LR}{\mathscr R}
\newcommand{\LS}{\mathscr S}

\newcommand{\LV}{\mathscr V}

\newcommand{\BCP}{\boldsymbol{\CP}}



%% file: special_commands.tex
\newcommand{\pa}{\psi_{\ba}}
\newcommand{\oa}{\omega_{\ba}}
\newcommand{\ha}{h_{\ba}}
\newcommand{\ca}{\vartheta_{\ba}}
\newcommand{\Ga}{{\Gamma_{\ba}}}

\newcommand{\cK}{\vartheta_K}

\newcommand{\diva}{d^{\ba}}
\newcommand{\bnda}{b^{\ba}}

\newcommand{\Ctra}{C_{{\rm tr},\ba}}
\newcommand{\CPa}{C_{{\rm P},\ba}}
\newcommand{\Cconta}{C_{{\rm cont},\ba}}
\newcommand{\Csta}{C_{{\rm st},\ba}}
\newcommand{\Cst}{C_{\rm st}}
\newcommand{\Cgeo}{C_{\rm geo}}
\newcommand{\Cgeoa}{C_{{\rm geo},\ba}}
\newcommand{\Clb}{C_{\rm lb}}

\newcommand{\hs}{h_\star}
\newcommand{\cs}{\vartheta_\star}

\newcommand{\hmax}{h_{\max}}
\newcommand{\cmin}{\vartheta_{\min}}

\newcommand{\osc}{\operatorname{osc}}

\newcommand{\Ci}{C_{\rm i}}

\newcommand{\bchi}{\boldsymbol \chi}
\newcommand{\dt}{\delta t}

\newcommand{\sr}{\beta}
\newcommand{\supp}{\operatorname{supp}}

%% file: slope_triangle.tex
\newcommand{\SlopeTriangle}[6]
{

    \pgfplotsextra
    {
        \pgfkeysgetvalue{/pgfplots/xmin}{\xmin}
        \pgfkeysgetvalue{/pgfplots/xmax}{\xmax}
        \pgfkeysgetvalue{/pgfplots/ymin}{\ymin}
        \pgfkeysgetvalue{/pgfplots/ymax}{\ymax}

        \pgfmathsetmacro{\xArel}{#1}
        \pgfmathsetmacro{\yArel}{#3}
        \pgfmathsetmacro{\xBrel}{#1-#2}
        \pgfmathsetmacro{\yBrel}{\yArel}
        \pgfmathsetmacro{\xCrel}{\xArel}

        \pgfmathsetmacro{\lnxB}{\xmin*(1-(#1-#2))+\xmax*(#1-#2)} 
        \pgfmathsetmacro{\lnxA}{\xmin*(1-#1)+\xmax*#1} 
        \pgfmathsetmacro{\lnyA}{\ymin*(1-#3)+\ymax*#3} 
        \pgfmathsetmacro{\lnyC}{\lnyA+#4*(\lnxA-\lnxB)}
        \pgfmathsetmacro{\yCrel}{\lnyC-\ymin)/(\ymax-\ymin)} 

        \coordinate (A) at (rel axis cs:\xArel,\yArel);
        \coordinate (B) at (rel axis cs:\xBrel,\yBrel);
        \coordinate (C) at (rel axis cs:\xCrel,\yCrel);

        \draw[#6]   (A)-- node[anchor=north] {#5}
                    (B)--
                    (C)--
                    cycle;
    }
}

%% file: figures/standing/history.tex
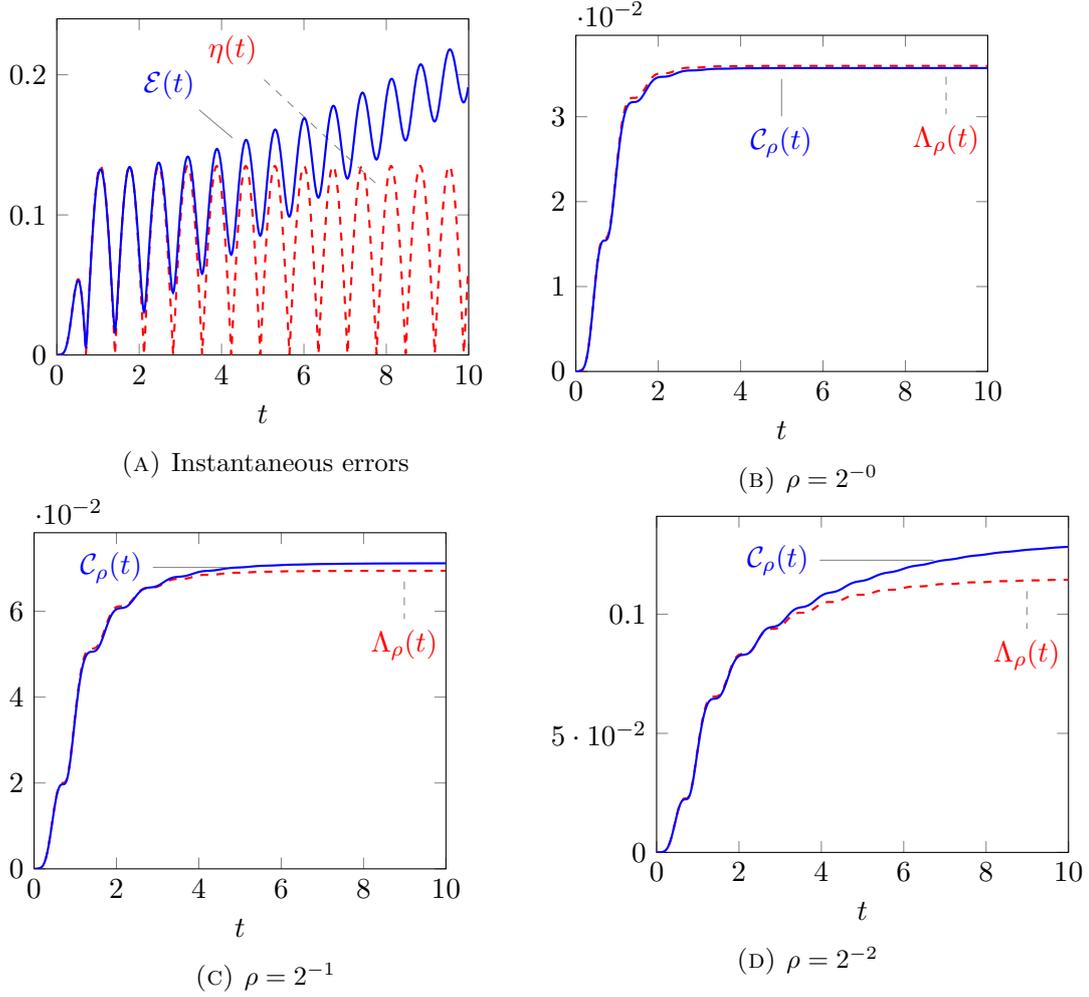
\begin{figure}

\begin{minipage}{.45\linewidth}
\begin{tikzpicture}
\begin{axis}%
[
	width  = \linewidth,
	xlabel = $t$,
	xmin=0,
	xmax=10,
	ymin=0
]

\plot[thick,red,dashed] table[x=t,y=est] {figures/standing/data/inst.txt}
node[pos=.80,pin={[pin distance=2cm]135:{$\eta(t)$}}] {};
\plot[thick,blue,solid] table[x=t,y=err] {figures/standing/data/inst.txt}
node[pos=.46,pin=135:{$\CE(t)$}] {};

\end{axis}
\end{tikzpicture}
\subcaption{Instantaneous errors}
\label{figure_standing_history_inst}
\end{minipage}
\begin{minipage}{.45\linewidth}
\begin{tikzpicture}
\begin{axis}%
[
	width  = \linewidth,
	xlabel = {$t$},
	xmin=0,
	xmax=10,
	ymin=0
]

\plot[thick,red,dashed] table[x=t,y=est] {figures/standing/data/cumu_1.0000.txt}
node[pos=.9,pin=-90:{$\Lambda_\rho(t)$}] {};
\plot[thick,blue,solid] table[x=t,y=err] {figures/standing/data/cumu_1.0000.txt}
node[pos=.5,pin={-90:{$\CC_\rho(t)$}}] {};

\end{axis}
\end{tikzpicture}
\subcaption{$\rho = 2^{-0}$}
\label{figure_standing_history_cumu_1.0000}
\end{minipage}

\begin{minipage}{.45\linewidth}
\begin{tikzpicture}
\begin{axis}%
[
	width  = \linewidth,
	xlabel = $t$,
	xmin=0,
	xmax=10,
	ymin=0
]

\plot[thick,red,dashed] table[x=t,y=est] {figures/standing/data/cumu_0.5000.txt}
node[pos=.9,pin=-90:{$\Lambda_\rho(t)$}] {};
\plot[thick,blue,solid] table[x=t,y=err] {figures/standing/data/cumu_0.5000.txt}
node[pos=.5,pin={[pin distance=1cm]180:{$\CC_\rho(t)$}}] {};

\end{axis}
\end{tikzpicture}
\subcaption{$\rho = 2^{-1}$}
\label{figure_standing_history_cumu_0.5000}
\end{minipage}
\begin{minipage}{.45\linewidth}
\begin{tikzpicture}
\begin{axis}%
[
	width  = \linewidth,
	xlabel = {$t$},
	xmin=0,
	xmax=10,
	ymin=0
]

\plot[thick,red,dashed] table[x=t,y=est] {figures/standing/data/cumu_0.2500.txt}
node[pos=.9,pin=-90:{$\Lambda_\rho(t)$}] {};
\plot[thick,blue,solid] table[x=t,y=err] {figures/standing/data/cumu_0.2500.txt}
node[pos=.7,pin={[pin distance=1.5cm]180:{$\CC_\rho(t)$}}] {};

\end{axis}
\end{tikzpicture}
\subcaption{$\rho = 2^{-2}$}
\label{figure_standing_history_cumu_0.2500}
\end{minipage}

\caption{Error evolution in the standing wave example for $h_{\max}=0.05$}
\label{figure_standing_history}
\end{figure}

%% file: figures/standing/convergence.tex
\begin{figure}

\begin{minipage}{.45\linewidth}
\begin{tikzpicture}
\begin{axis}%
[
	width  = .95\linewidth,
	xlabel = $N_{\rm dofs}^{1/2}$,
	ylabel = {Errors},
	xmode  = log,
	ymode  = log
]

\plot table[x expr=sqrt(\thisrow{nr_dofs}),y=est] {figures/standing/data/P0_R0.5000.txt}
node[pos=.1,pin=-90:{$\Lambda_\rho$}] {};
\plot table[x expr=sqrt(\thisrow{nr_dofs}),y=err] {figures/standing/data/P0_R0.5000.txt}
node[pos=.1,pin=0:{$\enorm{u-u_h}_\rho$}] {};

\plot[black,solid,mark=none,domain=10:90] {x^(-1)};
\SlopeTriangle{.50}{-.15}{.25}{-1}{$N_{\rm dofs}^{-1/2}$}{}

\end{axis}
\end{tikzpicture}
\subcaption{$p=1$, $\alpha=1.5$}
\label{figure_standing_convergence_P0}
\end{minipage}
\begin{minipage}{.45\linewidth}
\begin{tikzpicture}
\begin{axis}%
[
	width  = .95\linewidth,
	xlabel = $N_{\rm dofs}^{1/2}$,
	xmode  = log,
	ymode  = log
]

\plot table[x expr=sqrt(\thisrow{nr_dofs}),y=est] {figures/standing/data/P1_R0.5000.txt}
node[pos=.1,pin=-90:{$\Lambda_\rho$}] {};
\plot table[x expr=sqrt(\thisrow{nr_dofs}),y=err] {figures/standing/data/P1_R0.5000.txt}
node[pos=.1,pin=0:{$\enorm{u-u_h}_\rho$}] {};

\plot[black,solid,mark=none,domain=30:200] {2*x^(-2)};
\SlopeTriangle{.50}{-.15}{.25}{-2}{$N_{\rm dofs}^{-1}$}{}

\end{axis}
\end{tikzpicture}
\subcaption{$p=2$, $\alpha=0.1$}
\label{figure_standing_convergence_P1}
\end{minipage}

\caption{Convergence in the standing wave example for $\rho=0.5$}
\label{figure_standing_convergence}
\end{figure}
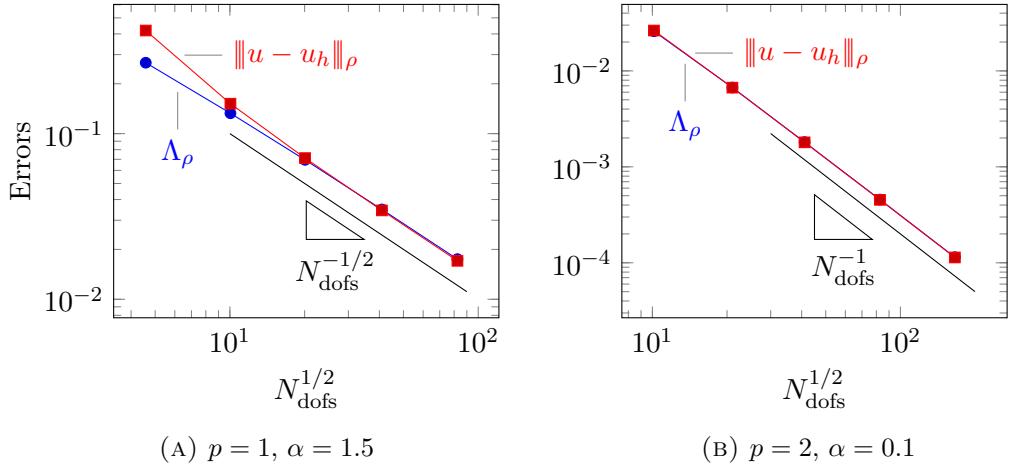

%% file: figures/standing/efficiency.tex
\begin{figure}

\begin{minipage}{.45\linewidth}
\begin{tikzpicture}
\begin{axis}%
[
	width  = .95\linewidth,
	xlabel = $N_{\rm dofs}^{1/2}$,
	ylabel = {Effictivity index},
	xmode  = log,
	ymin   = 0.35,
	ymax   = 1.10
]

\plot table[x expr=sqrt(\thisrow{nr_dofs}),y expr=\thisrow{est}/\thisrow{err}]
{figures/standing/data/P0_R1.0000.txt}
node[pos=.1,pin=90:{$\rho = 1$}] {};

\plot table[x expr=sqrt(\thisrow{nr_dofs}),y expr=\thisrow{est}/\thisrow{err}]
{figures/standing/data/P0_R0.5000.txt}
node[pos=.15,pin={[pin distance=2cm]0:{$\rho = 0.5$}}] {};

\plot table[x expr=sqrt(\thisrow{nr_dofs}),y expr=\thisrow{est}/\thisrow{err}]
{figures/standing/data/P0_R0.2500.txt}
node[pos=.1,pin=0:{$\rho = 0.25$}] {};

\end{axis}
\end{tikzpicture}
\subcaption{$p=1$, $\alpha=1.5$}
\label{figure_standing_efficiency_P0}
\end{minipage}
\begin{minipage}{.45\linewidth}
\begin{tikzpicture}
\begin{axis}%
[
	width  = .95\linewidth,
	xlabel = $N_{\rm dofs}^{1/2}$,
	xmode  = log,
	ymin   = 0.35,
	ymax   = 1.10
]

\plot table[x expr=sqrt(\thisrow{nr_dofs}),y expr=\thisrow{est}/\thisrow{err}]
{figures/standing/data/P1_R1.0000.txt}
node[pos=.1,pin={[pin distance=.1cm]90:{$\rho = 1$}}] {};

\plot table[x expr=sqrt(\thisrow{nr_dofs}),y expr=\thisrow{est}/\thisrow{err}]
{figures/standing/data/P1_R0.5000.txt}
node[pos=.1,pin={[pin distance=2cm]-90:{$\rho = 0.5$}}] {};

\plot table[x expr=sqrt(\thisrow{nr_dofs}),y expr=\thisrow{est}/\thisrow{err}]
{figures/standing/data/P1_R0.2500.txt}
node[pos=.15,pin=-45:{$\rho = 0.25$}] {};

\end{axis}
\end{tikzpicture}
\subcaption{$p=2$, $\alpha=0.1$}
\label{figure_standing_efficiency_P1}
\end{minipage}

\begin{minipage}{.45\linewidth}
\begin{tikzpicture}
\begin{axis}%
[
	width  = .95\linewidth,
	xlabel = $N_{\rm dofs}^{1/2}$,
	ylabel = {Effictivity index},
	xmode  = log,
	ymin   = 0.35,
	ymax   = 1.10
]

\plot table[x expr=sqrt(\thisrow{nr_dofs}),y expr=\thisrow{est}/\thisrow{err}]
{figures/standing/data/P1_cfl_R1.0000.txt}
node[pos=.1,pin={[pin distance=.1cm]90:{$\rho = 1$}}] {};

\plot table[x expr=sqrt(\thisrow{nr_dofs}),y expr=\thisrow{est}/\thisrow{err}]
{figures/standing/data/P1_cfl_R0.5000.txt}
node[pos=.1,pin={[pin distance=2cm]-90:{$\rho = 0.5$}}] {};

\plot table[x expr=sqrt(\thisrow{nr_dofs}),y expr=\thisrow{est}/\thisrow{err}]
{figures/standing/data/P1_cfl_R0.2500.txt}
node[pos=.15,pin=-45:{$\rho = 0.25$}] {};

\end{axis}
\end{tikzpicture}
\subcaption{$p=2$, $\alpha=0.6$}
\label{figure_standing_efficiency_P1_cfl}
\end{minipage}
\begin{minipage}{.45\linewidth}
$ $
\end{minipage}

\caption{Effectivity index in the standing wave example}
\label{figure_standing_efficiency}
\end{figure}
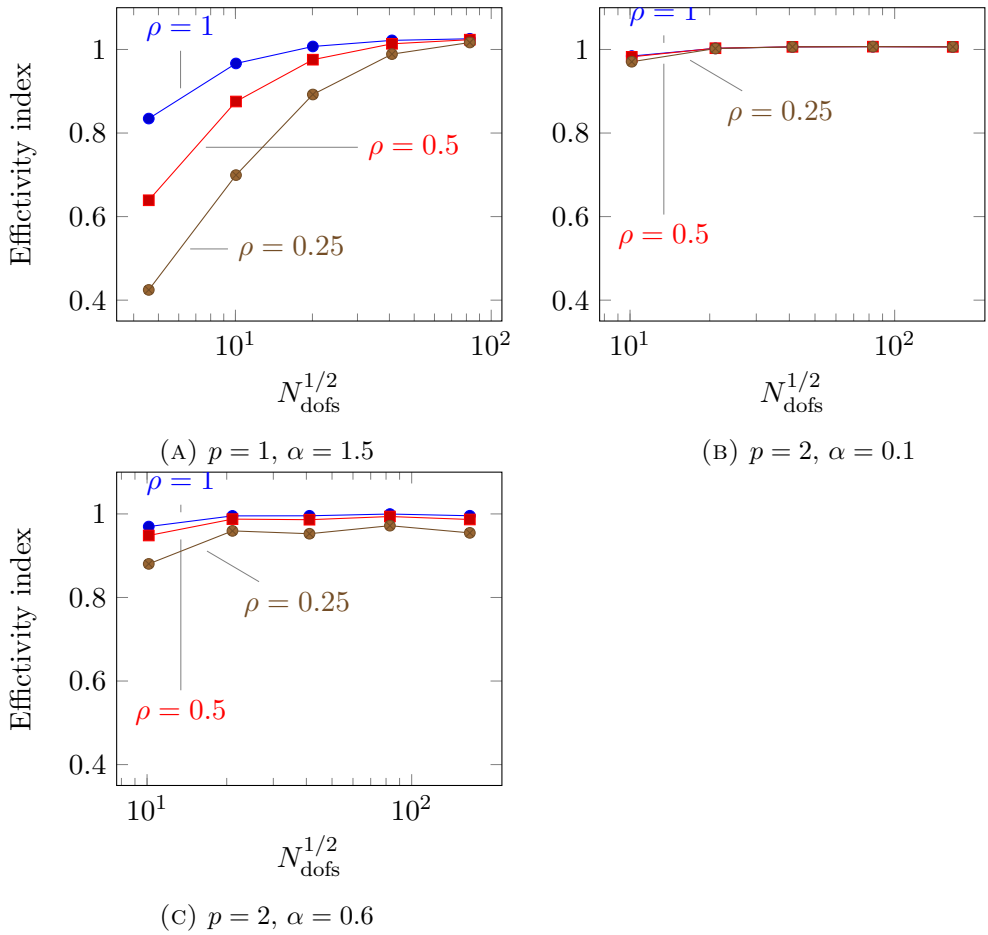

%% file: figures/reflection/history.tex
\begin{figure}

\begin{minipage}{.45\linewidth}
\begin{tikzpicture}
\begin{axis}%
[
	width  = \linewidth,
	xlabel = $t$,
	xmin=0,
	xmax=10,
	ymin=0
]

\plot[thick,red,dashed] table[x=t,y=est] {figures/reflection/data/inst.txt}
node[pos=.525,pin=45:{$\eta(t)$}] {};
\plot[thick,blue,solid] table[x=t,y=err] {figures/reflection/data/inst.txt}
node[pos=.5,pin=45:{$\CE_\rho(t)$}] {};

\end{axis}
\end{tikzpicture}
\subcaption{Instantaneous errors}
\label{figure_reflection_history_inst}
\end{minipage}
\begin{minipage}{.45\linewidth}
\begin{tikzpicture}
\begin{axis}%
[
	width  = \linewidth,
	xlabel = {$t$},
	xmin=0,
	xmax=10,
	ymin=0
]

\plot[thick,red,dashed] table[x=t,y=est] {figures/reflection/data/cumu.txt}
node[pos=.8,pin=-90:{$\Lambda_\rho(t)$}] {};
\plot[thick,blue,solid] table[x=t,y=err] {figures/reflection/data/cumu.txt}
node[pos=.5,pin={180:{$\CC_\rho(t)$}}] {};

\end{axis}
\end{tikzpicture}
\subcaption{Cumulated errors}
\label{figure_reflection_history_1.0000}
\end{minipage}

\caption{Error evolution in the reflection wave example for $\sigma=0.5$, $p=2$, $\alpha=0.6$ and $h_{\max}=0.05$}
\label{figure_reflection_history}
\end{figure}
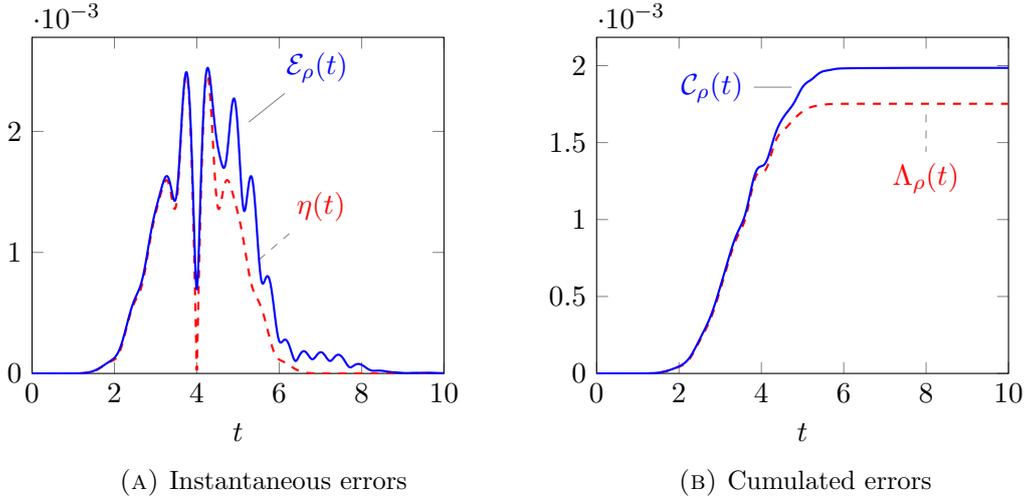

%% file: figures/reflection/convergence.tex
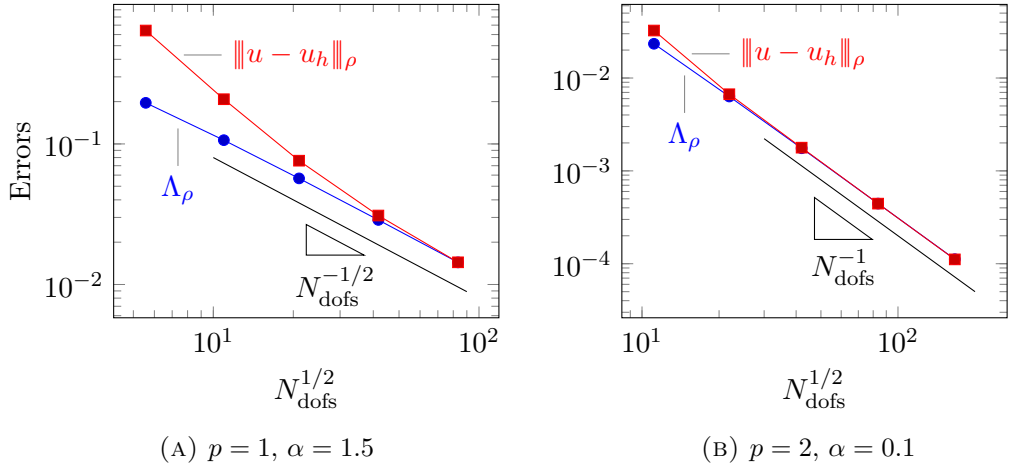
\begin{figure}

\begin{minipage}{.45\linewidth}
\begin{tikzpicture}
\begin{axis}%
[
	width  = .95\linewidth,
	xlabel = $N_{\rm dofs}^{1/2}$,
	ylabel = {Errors},
	xmode  = log,
	ymode  = log
]

\plot table[x expr=sqrt(\thisrow{nr_dofs}),y=est] {figures/reflection/data/P0_S0.5000.txt}
node[pos=.1,pin=-90:{$\Lambda_\rho$}] {};
\plot table[x expr=sqrt(\thisrow{nr_dofs}),y=err] {figures/reflection/data/P0_S0.5000.txt}
node[pos=.1,pin=0:{$\enorm{u-u_h}_\rho$}] {};

\plot[black,solid,mark=none,domain=10:90] {.8*x^(-1)};
\SlopeTriangle{.50}{-.15}{.20}{-1}{$N_{\rm dofs}^{-1/2}$}{}

\end{axis}
\end{tikzpicture}
\subcaption{$p=1$, $\alpha=1.5$}
\label{figure_reflection_convergence_P0}
\end{minipage}
\begin{minipage}{.45\linewidth}
\begin{tikzpicture}
\begin{axis}%
[
	width  = .95\linewidth,
	xlabel = $N_{\rm dofs}^{1/2}$,
	xmode  = log,
	ymode  = log
]

\plot table[x expr=sqrt(\thisrow{nr_dofs}),y=est] {figures/reflection/data/P1_S0.5000.txt}
node[pos=.1,pin=-90:{$\Lambda_\rho$}] {};
\plot table[x expr=sqrt(\thisrow{nr_dofs}),y=err] {figures/reflection/data/P1_S0.5000.txt}
node[pos=.1,pin=0:{$\enorm{u-u_h}_\rho$}] {};

\plot[black,solid,mark=none,domain=30:200] {2*x^(-2)};
\SlopeTriangle{.50}{-.15}{.25}{-2}{$N_{\rm dofs}^{-1}$}{}

\end{axis}
\end{tikzpicture}
\subcaption{$p=2$, $\alpha=0.1$}
\label{figure_reflection_convergence_P1}
\end{minipage}

\caption{Convergence in the reflection wave example for $\sigma=0.5$}
\label{figure_reflection_convergence}
\end{figure}

%% file: figures/reflection/efficiency.tex
\begin{figure}

\begin{minipage}{.45\linewidth}
\begin{tikzpicture}
\begin{axis}%
[
	width  = .95\linewidth,
	xlabel = $N_{\rm dofs}^{1/2}$,
	ylabel = {Effictivity index},
	xmode  = log
]

\plot table[x expr=sqrt(\thisrow{nr_dofs},y expr=\thisrow{est}/\thisrow{err}]
{figures/reflection/data/P0_S0.5000.txt}
node[pos=.6,pin=135:{$\sigma = 2^{-1}$}] {};

\plot table[x expr=sqrt(\thisrow{nr_dofs},y expr=\thisrow{est}/\thisrow{err}]
{figures/reflection/data/P0_S0.2500.txt}
node[pos=.1,pin={[pin distance=2cm]90:{$\sigma = 2^{-2}$}}] {};

\plot table[x expr=sqrt(\thisrow{nr_dofs},y expr=\thisrow{est}/\thisrow{err}]
{figures/reflection/data/P0_S0.1250.txt}
node[pos=.9,pin=90:{$\sigma = 2^{-3}$}] {};

\end{axis}
\end{tikzpicture}
\subcaption{$p=1$, $\alpha=1.5$}
\label{figure_reflection_efficiency_P0}
\end{minipage}
\begin{minipage}{.45\linewidth}
\begin{tikzpicture}
\begin{axis}%
[
	width  = .95\linewidth,
	xlabel = $N_{\rm dofs}^{1/2}$,
	xmode  = log
]

\plot table[x expr=sqrt(\thisrow{nr_dofs},y expr=\thisrow{est}/\thisrow{err}]
{figures/reflection/data/P1_S0.5000.txt}
node[pos=.1,pin=90:{$\sigma = 2^{-1}$}] {};

\plot table[x expr=sqrt(\thisrow{nr_dofs},y expr=\thisrow{est}/\thisrow{err}]
{figures/reflection/data/P1_S0.2500.txt}
node[pos=.1,pin=90:{$\sigma = 2^{-2}$}] {};

\plot table[x expr=sqrt(\thisrow{nr_dofs},y expr=\thisrow{est}/\thisrow{err}]
{figures/reflection/data/P1_S0.1250.txt}
node[pos=.9,pin=-90:{$\sigma = 2^{-3}$}] {};

\end{axis}
\end{tikzpicture}
\subcaption{$p=2$, $\alpha=0.1$}
\label{figure_reflection_efficiency_P1}
\end{minipage}

\begin{minipage}{.45\linewidth}
\begin{tikzpicture}
\begin{axis}%
[
	width  = .95\linewidth,
	ylabel = {Effictivity index},
	xlabel = $N_{\rm dofs}^{1/2}$,
	xmode  = log
]

\plot table[x expr=sqrt(\thisrow{nr_dofs},y expr=\thisrow{est}/\thisrow{err}]
{figures/reflection/data/P1_cfl_S0.5000.txt}
node[pos=.1,pin=90:{$\sigma = 2^{-1}$}] {};

\plot table[x expr=sqrt(\thisrow{nr_dofs},y expr=\thisrow{est}/\thisrow{err}]
{figures/reflection/data/P1_cfl_S0.2500.txt}
node[pos=.1,pin=90:{$\sigma = 2^{-2}$}] {};

\plot table[x expr=sqrt(\thisrow{nr_dofs},y expr=\thisrow{est}/\thisrow{err}]
{figures/reflection/data/P1_cfl_S0.1250.txt}
node[pos=.9,pin=-90:{$\sigma = 2^{-3}$}] {};

\end{axis}
\end{tikzpicture}
\subcaption{$p=2$, $\alpha=0.6$}
\label{figure_reflection_efficiency_P1_cfl}
\end{minipage}
\begin{minipage}{.45\linewidth}
$ $
\end{minipage}

\caption{Effectivity index in the reflection example}
\label{figure_reflection_efficiency}
\end{figure}
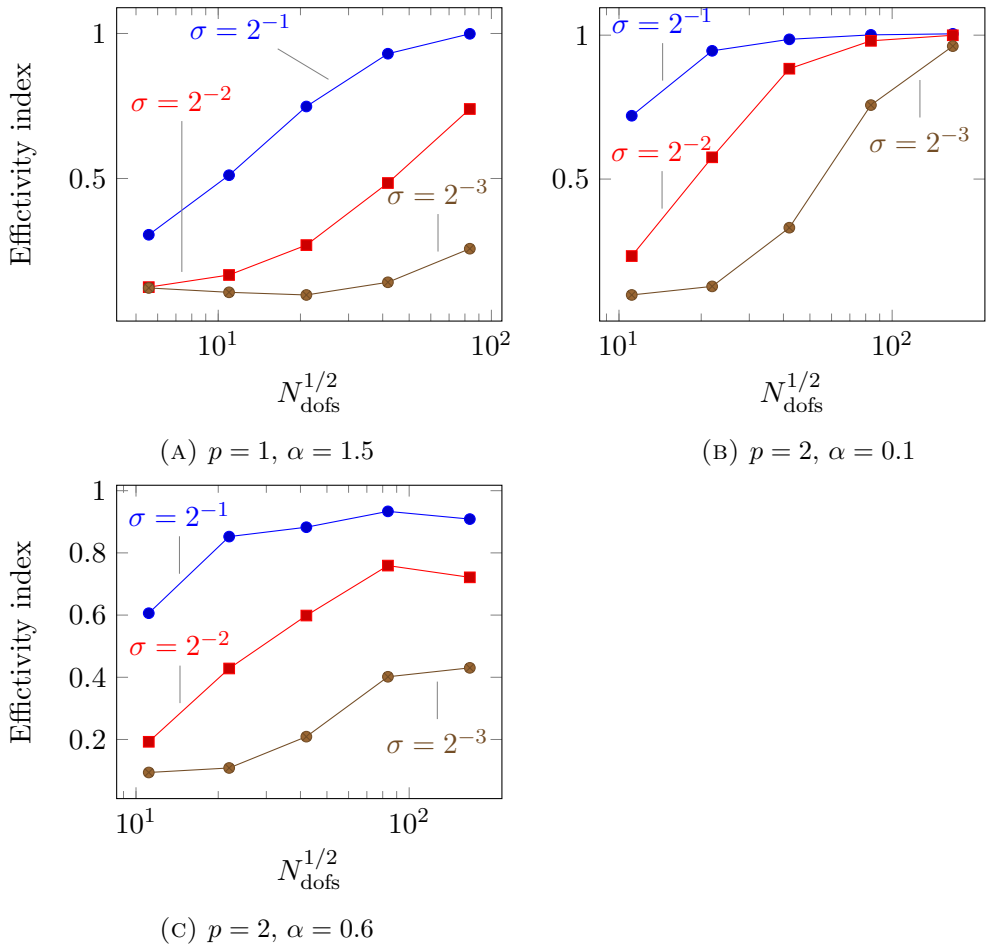

%% file: figures/reflection/pictures.tex
\begin{figure}

\begin{minipage}{.02\linewidth}
\begin{turn}{90}
Solution
\end{turn}
\end{minipage}
\begin{minipage}{.30\linewidth}
\begin{tikzpicture}
\draw (0,0) node[anchor=south west,inner sep=0]%
{\includegraphics[width=4.0cm]{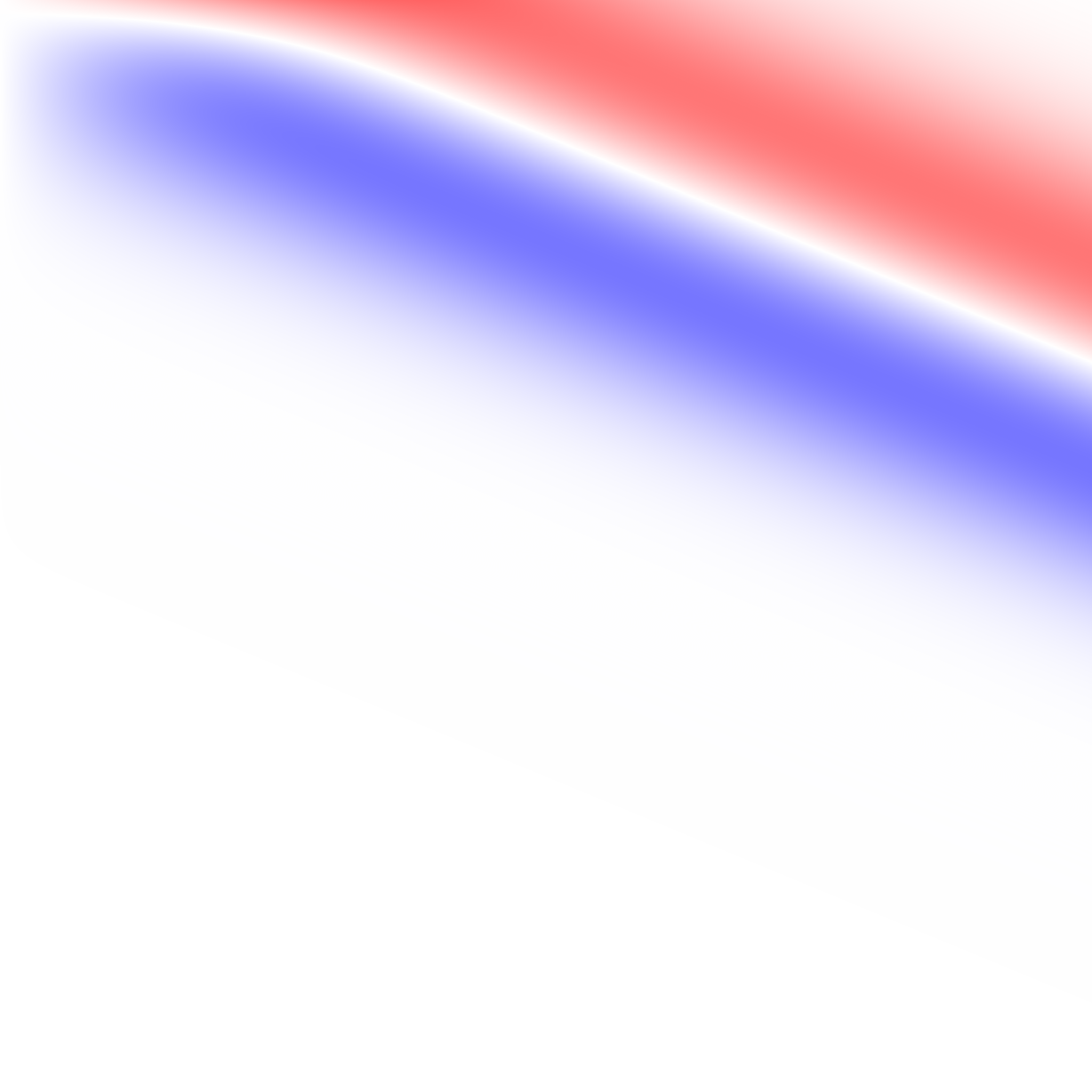}};
\draw[thick]        (4.0,0.0) -- (0.0,0.0) -- (0.0,4.0);
\draw[thick,dashed] (4.0,0.0) -- (4.0,4.0) -- (0.0,4.0);
\end{tikzpicture}
\end{minipage}
\begin{minipage}{.30\linewidth}
\begin{tikzpicture}
\draw (0,0) node[anchor=south west,inner sep=0]%
{\includegraphics[width=4.0cm]{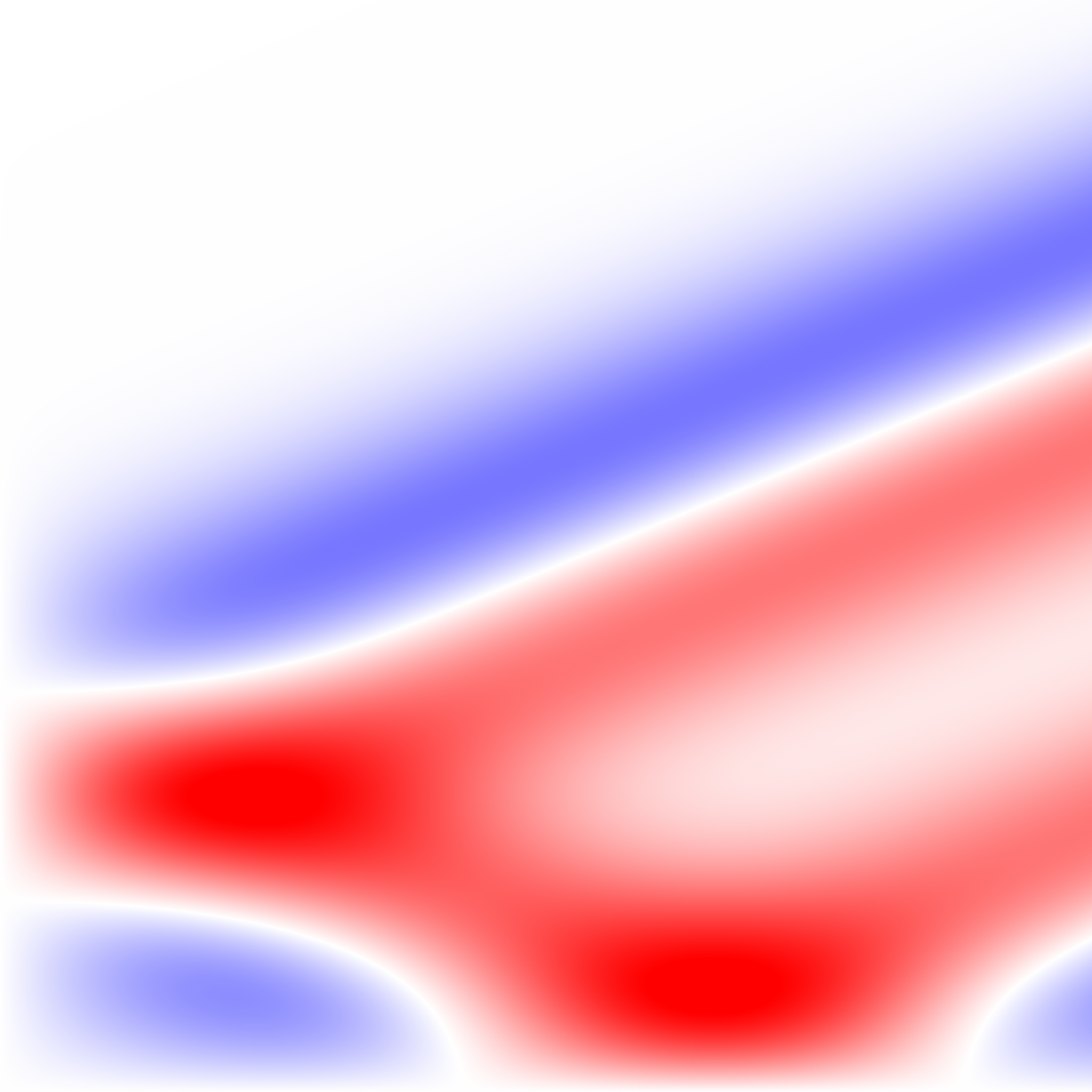}};
\draw[thick]        (4.0,0.0) -- (0.0,0.0) -- (0.0,4.0);
\draw[thick,dashed] (4.0,0.0) -- (4.0,4.0) -- (0.0,4.0);
\end{tikzpicture}
\end{minipage}
\begin{minipage}{.30\linewidth}
\begin{tikzpicture}
\draw (0,0) node[anchor=south west,inner sep=0]%
{\includegraphics[width=4.0cm]{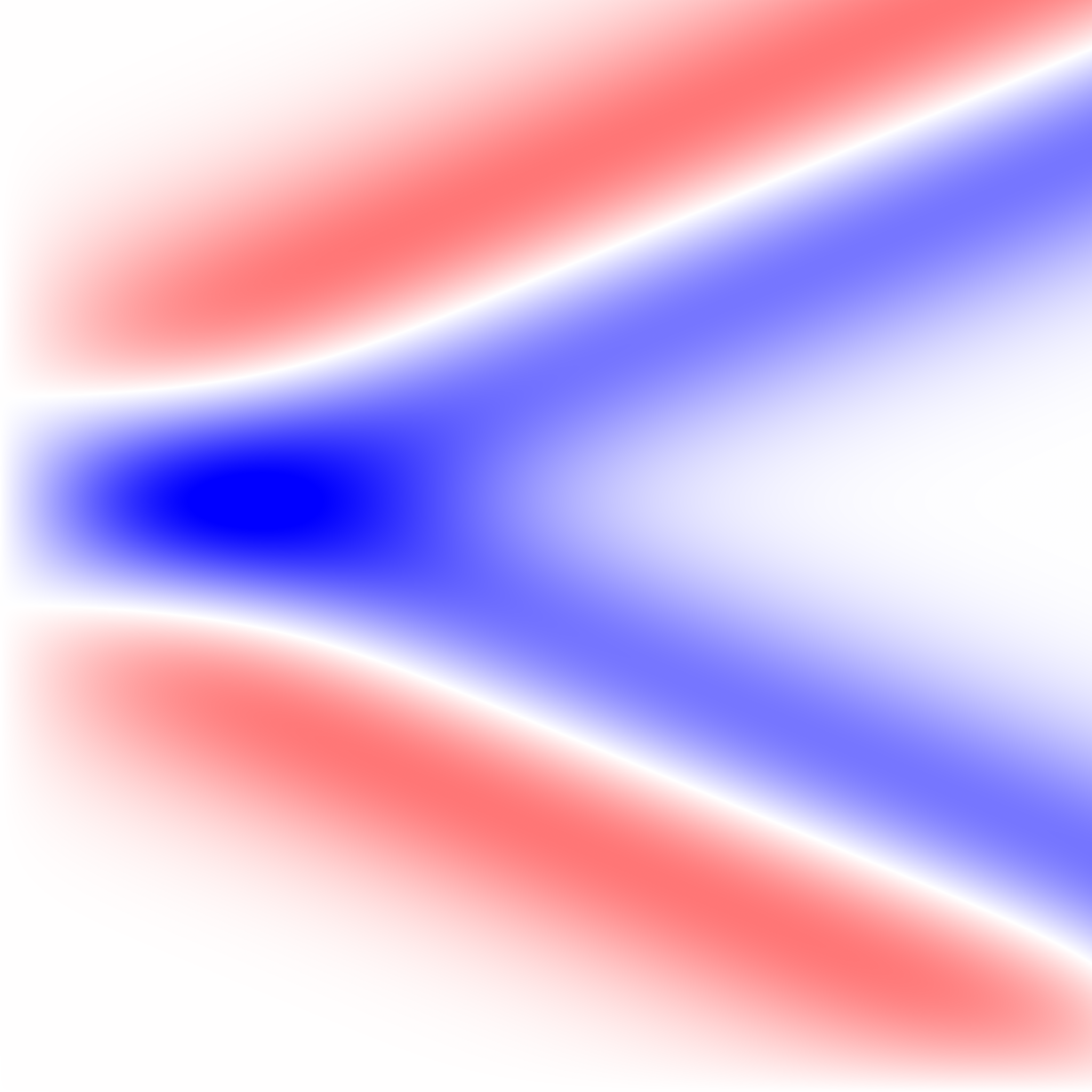}};
\draw[thick]        (4.0,0.0) -- (0.0,0.0) -- (0.0,4.0);
\draw[thick,dashed] (4.0,0.0) -- (4.0,4.0) -- (0.0,4.0);
\end{tikzpicture}
\end{minipage}
\begin{minipage}{.05\linewidth}
\begin{tikzpicture}
\draw (0,0) node[anchor=south west,inner sep=0]%
{\includegraphics[width=0.5cm,height=4.0cm]{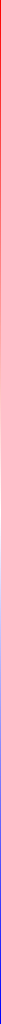}};
\draw[thick] (0,0) rectangle (.5,4.0);

\draw (0.5,0.00) node [anchor=west] {$-0.1$};
\draw (0.5,2.00) node [anchor=west] {${\color{white}+}0.0$};
\draw (0.5,4.00) node [anchor=west] {${\color{white}+}0.1$};
\end{tikzpicture}
\end{minipage}

\begin{minipage}{.02\linewidth}
\begin{turn}{90}
True error
\end{turn}
\end{minipage}
\begin{minipage}{.30\linewidth}
\begin{tikzpicture}
\draw (0,0) node[anchor=south west,inner sep=0]%
{\includegraphics[width=4.0cm]{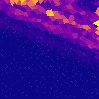}};
\draw[thick] (0,0) rectangle (4.0,4.0);
\end{tikzpicture}
\end{minipage}
\begin{minipage}{.30\linewidth}
\begin{tikzpicture}
\draw (0,0) node[anchor=south west,inner sep=0]%
{\includegraphics[width=4.0cm]{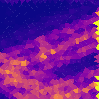}};
\draw[thick] (0,0) rectangle (4.0,4.0);
\end{tikzpicture}
\end{minipage}
\begin{minipage}{.30\linewidth}
\begin{tikzpicture}
\draw (0,0) node[anchor=south west,inner sep=0]%
{\includegraphics[width=4.0cm]{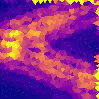}};
\draw[thick] (0,0) rectangle (4.0,4.0);
\end{tikzpicture}
\end{minipage}
\begin{minipage}{.05\linewidth}
\begin{tikzpicture}
\draw (0,0) node[anchor=south west,inner sep=0]%
{\includegraphics[width=0.5cm,height=4.0cm]{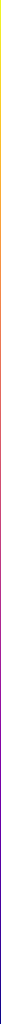}};
\draw[thick] (0,0) rectangle (.5,4.0);

\draw (0.5,0.00) node [anchor=west] {$0.000$};
\draw (0.5,4.00) node [anchor=west] {$0.002$};
\end{tikzpicture}
\end{minipage}

\begin{minipage}{.02\linewidth}
\begin{turn}{90}
Estimated error
\end{turn}

\begin{center}
$ $
\end{center}

\end{minipage}
\begin{minipage}{.30\linewidth}
\begin{tikzpicture}
\draw (0,0) node[anchor=south west,inner sep=0]%
{\includegraphics[width=4.0cm]{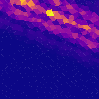}};
\draw[thick] (0,0) rectangle (4.0,4.0);
\end{tikzpicture}

\begin{center}
$t = 3.00$
\end{center}

\end{minipage}
\begin{minipage}{.30\linewidth}
\begin{tikzpicture}
\draw (0,0) node[anchor=south west,inner sep=0]%
{\includegraphics[width=4.0cm]{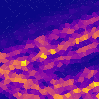}};
\draw[thick] (0,0) rectangle (4.0,4.0);
\end{tikzpicture}

\begin{center}
$t = 3.75$
\end{center}

\end{minipage}
\begin{minipage}{.30\linewidth}
\begin{tikzpicture}
\draw (0,0) node[anchor=south west,inner sep=0]%
{\includegraphics[width=4.0cm]{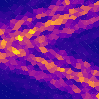}};
\draw[thick] (0,0) rectangle (4.0,4.0);
\end{tikzpicture}

\begin{center}
$t = 4.00$
\end{center}

\end{minipage}
\begin{minipage}{.05\linewidth}
\begin{tikzpicture}
\draw (0,0) node[anchor=south west,inner sep=0]%
{\includegraphics[width=0.5cm,height=4.0cm]{figures/reflection/images/cbar_plasma}};
\draw[thick] (0,0) rectangle (.5,4.0);

\draw (0.5,0.00) node [anchor=west] {$0.000$};
\draw (0.5,4.00) node [anchor=west] {$0.002$};
\end{tikzpicture}

$ $

\end{minipage}

\caption{Reflection example with $\sigma=2^{-3}$ $p=2$, $\alpha=0.6$, and $h=0.05$}
\label{figure_reflection_pictures}
\end{figure}

%% file: figures/penetrable/efficiency.tex
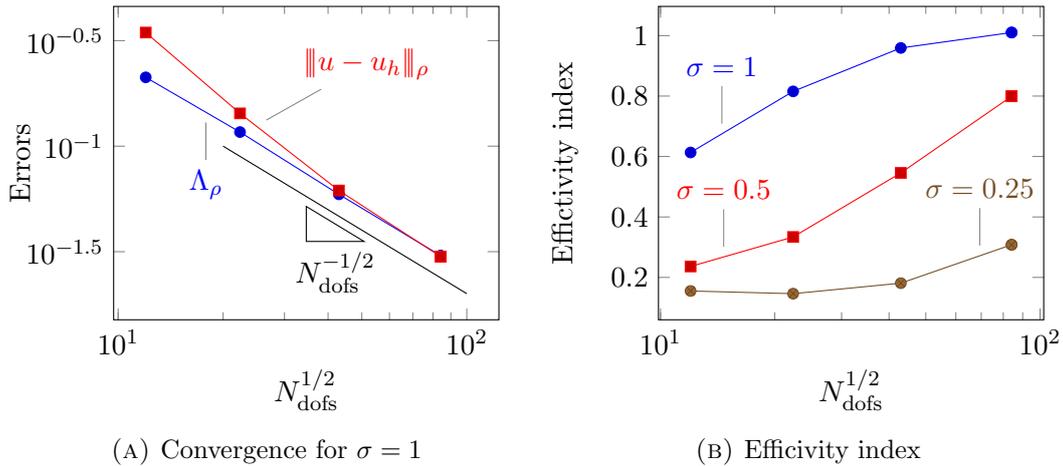
\begin{figure}

\begin{minipage}{.45\linewidth}
\begin{tikzpicture}
\begin{axis}%
[
	width  = .95\linewidth,
	xlabel = $N_{\rm dofs}^{1/2}$,
	ylabel = {Errors},
	xmode  = log,
	ymode  = log
]

\plot table[x expr=sqrt(\thisrow{nr_dofs},y=est]
{figures/penetrable/data/P0_S1.0000.txt}
node[pos=.2,pin=-90:{$\Lambda_\rho$}] {};

\plot table[x expr=sqrt(\thisrow{nr_dofs},y=err]
{figures/penetrable/data/P0_S1.0000.txt}
node[pos=.4,pin=45:{$\enorm{u-u_h}_\rho$}] {};

\plot[black,solid,mark=none,domain=20:100] {2*x^(-1)};
\SlopeTriangle{.50}{-.15}{.25}{-1}{$N_{\rm dofs}^{-1/2}$}{}

\end{axis}
\end{tikzpicture}
\subcaption{Convergence for $\sigma=1$}
\end{minipage}
\begin{minipage}{.45\linewidth}
\begin{tikzpicture}
\begin{axis}%
[
	width  = .95\linewidth,
	xlabel = $N_{\rm dofs}^{1/2}$,
	ylabel = {Effictivity index},
	xmode  = log
]

\plot table[x expr=sqrt(\thisrow{nr_dofs},y expr=\thisrow{est}/\thisrow{err}]
{figures/penetrable/data/P0_S1.0000.txt}
node[pos=.1,pin=90:{$\sigma=1$}] {};

\plot table[x expr=sqrt(\thisrow{nr_dofs},y expr=\thisrow{est}/\thisrow{err}]
{figures/penetrable/data/P0_S0.5000.txt}
node[pos=.1,pin=90:{$\sigma=0.5$}] {};

\plot table[x expr=sqrt(\thisrow{nr_dofs},y expr=\thisrow{est}/\thisrow{err}]
{figures/penetrable/data/P0_S0.2500.txt}
node[pos=.9,pin=90:{$\sigma=0.25$}] {};

\end{axis}
\end{tikzpicture}
\subcaption{Efficivity index}
\end{minipage}

\caption{Errors and effectivity index in the obstacle example}
\label{figure_penetrable_efficiency}
\end{figure}

%% file: figures/penetrable/pictures.tex
\begin{figure}

\begin{minipage}{.02\linewidth}
\begin{turn}{90}
Solution
\end{turn}
\end{minipage}
\begin{minipage}{.30\linewidth}
\begin{tikzpicture}
\draw (-2.00,-2.00) node[anchor=south west,inner sep=0]%
{\includegraphics[width=4.0cm]{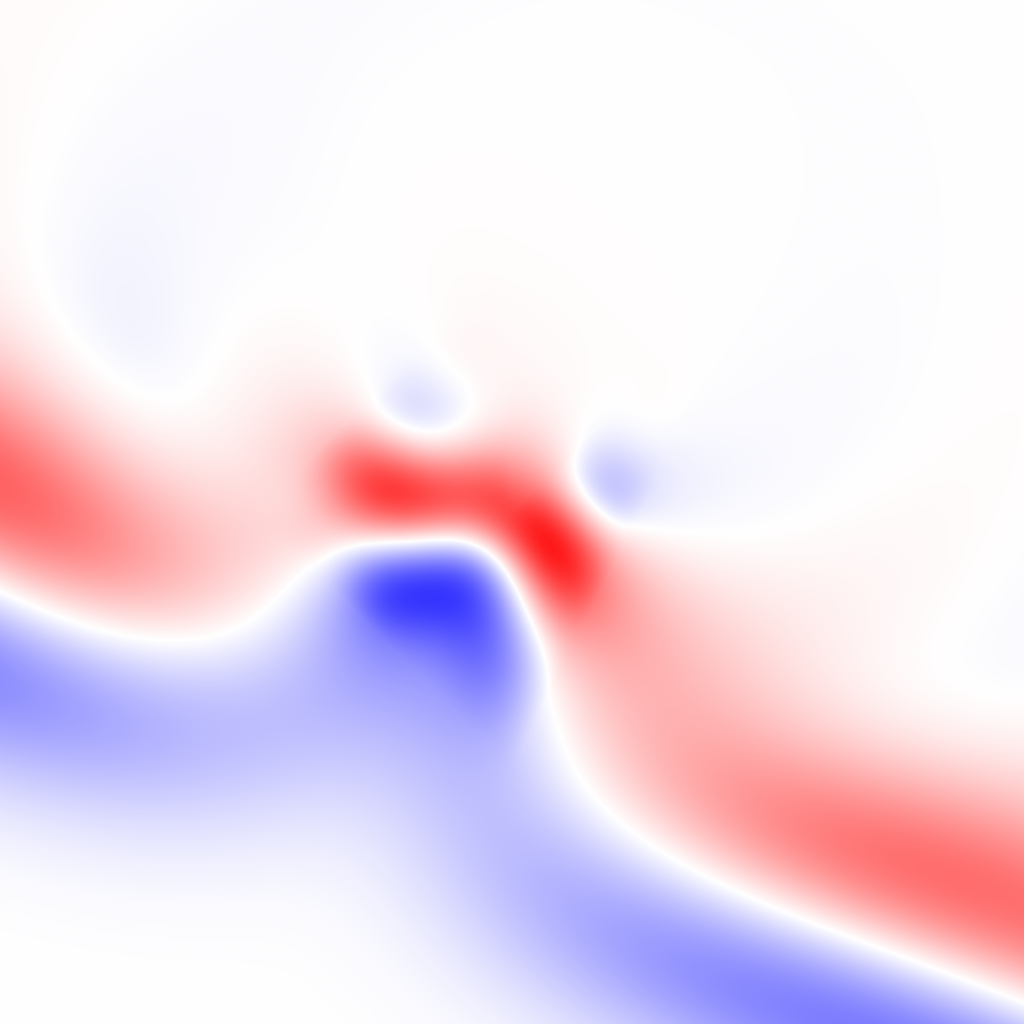}};
\draw[thick,dashed] (-2.00,-2.00) rectangle (2.00,2.00);

\draw[thick] ( 0.00,-1.12) -- ( 1.12, 1.12) -- ( 0.00, 0.00) -- ( 0.00, 1.12) -- (-1.12, 0.00) -- cycle;
\end{tikzpicture}
\end{minipage}
\begin{minipage}{.30\linewidth}
\begin{tikzpicture}
\draw (-2.00,-2.00) node[anchor=south west,inner sep=0]%
{\includegraphics[width=4.0cm]{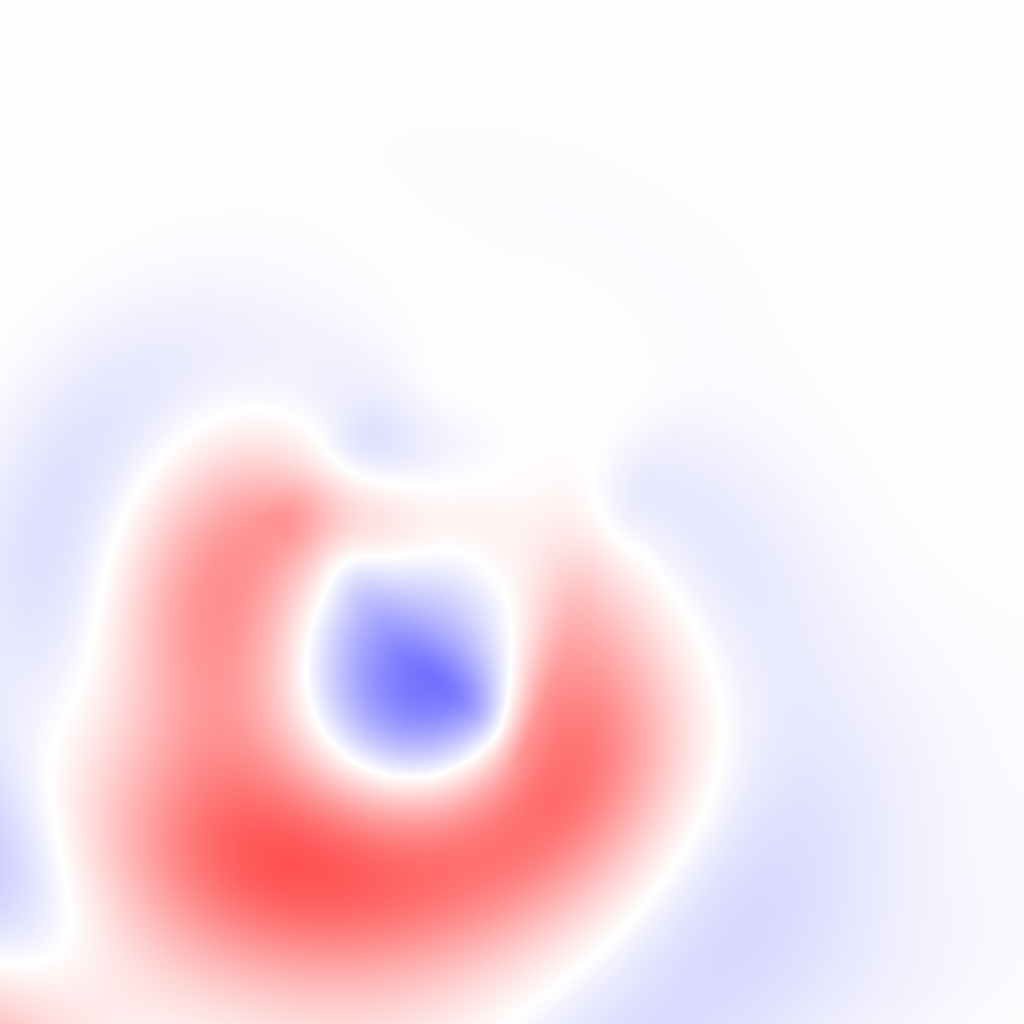}};
\draw[thick,dashed] (-2.00,-2.00) rectangle (2.00,2.00);

\draw[thick] ( 0.00,-1.12) -- ( 1.12, 1.12) -- ( 0.00, 0.00) -- ( 0.00, 1.12) -- (-1.12, 0.00) -- cycle;
\end{tikzpicture}
\end{minipage}
\begin{minipage}{.30\linewidth}
\begin{tikzpicture}
\draw (-2.00,-2.00) node[anchor=south west,inner sep=0]%
{\includegraphics[width=4.0cm]{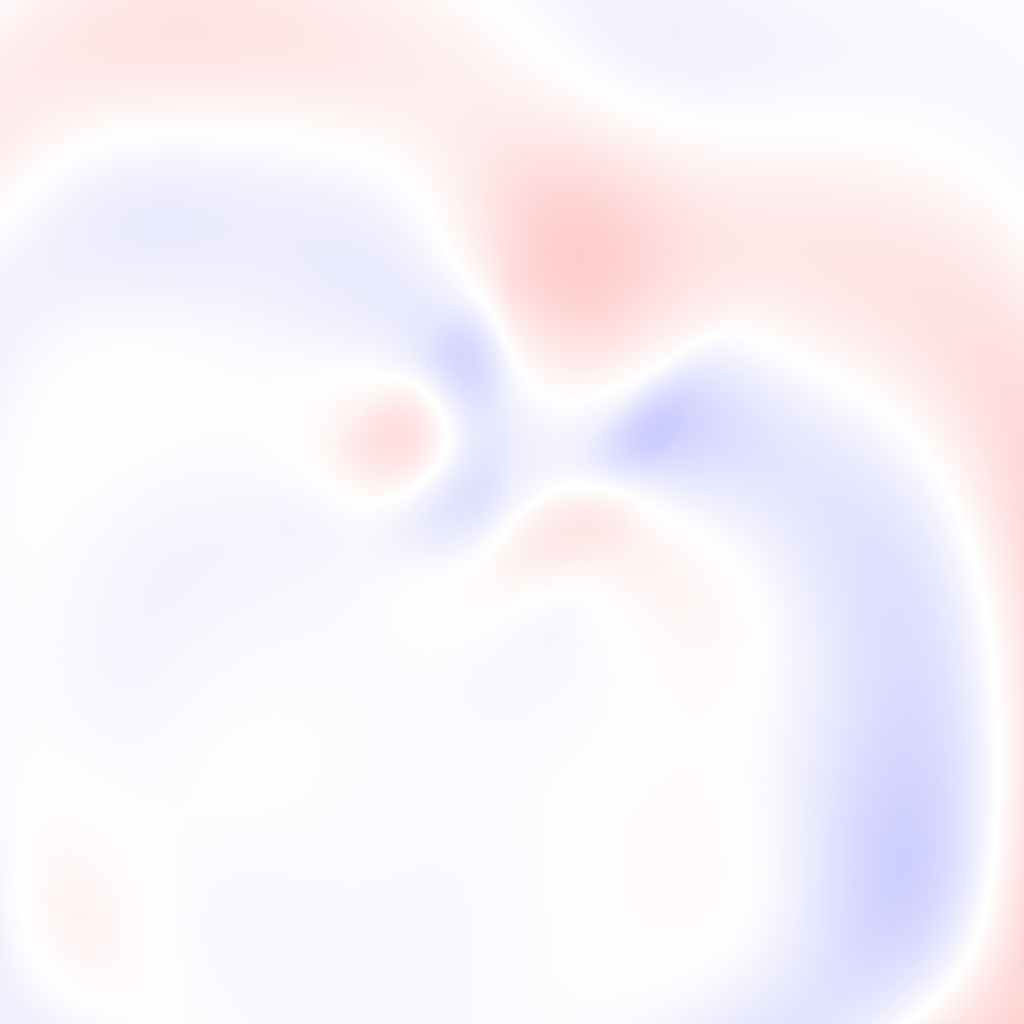}};
\draw[thick,dashed] (-2.00,-2.00) rectangle (2.00,2.00);

\draw[thick] ( 0.00,-1.12) -- ( 1.12, 1.12) -- ( 0.00, 0.00) -- ( 0.00, 1.12) -- (-1.12, 0.00) -- cycle;
\end{tikzpicture}
\end{minipage}
\begin{minipage}{.05\linewidth}
\begin{tikzpicture}
\draw (0,-2.00) node[anchor=south west,inner sep=0]%
{\includegraphics[width=0.5cm,height=4.0cm]{figures/reflection/images/cbar_bwr}};
\draw[thick] (0,-2.00) rectangle (.5,2.00);

\draw (0.5,-2.00) node [anchor=west] {$-0.2$};
\draw (0.5, 2.00) node [anchor=west] {${\color{white}+}0.2$};
\draw (0.5, 0   ) node [anchor=west] {${\color{white}+}0.0$};
\end{tikzpicture}
\end{minipage}

\begin{minipage}{.02\linewidth}
\begin{turn}{90}
True error
\end{turn}
\end{minipage}
\begin{minipage}{.30\linewidth}
\begin{tikzpicture}
\draw (0,0) node[anchor=south west,inner sep=0]%
{\includegraphics[width=4.0cm]{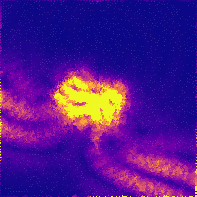}};
\draw[thick] (0,0) rectangle (4.0,4.0);
\end{tikzpicture}
\end{minipage}
\begin{minipage}{.30\linewidth}
\begin{tikzpicture}
\draw (0,0) node[anchor=south west,inner sep=0]%
{\includegraphics[width=4.0cm]{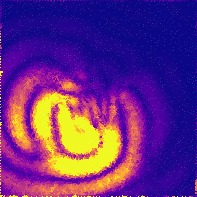}};
\draw[thick] (0,0) rectangle (4.0,4.0);
\end{tikzpicture}
\end{minipage}
\begin{minipage}{.30\linewidth}
\begin{tikzpicture}
\draw (0,0) node[anchor=south west,inner sep=0]%
{\includegraphics[width=4.0cm]{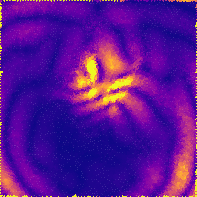}};
\draw[thick] (0,0) rectangle (4.0,4.0);
\end{tikzpicture}
\end{minipage}
\begin{minipage}{.05\linewidth}
\begin{tikzpicture}
\draw (0,-2.00) node[anchor=south west,inner sep=0]%
{\includegraphics[width=0.5cm,height=4.0cm]{figures/reflection/images/cbar_plasma}};
\draw[thick] (0,-2.00) rectangle (.5,2.00);
\draw[thick] (0,-2.00) rectangle (.5,2.00);

\draw (0.5,-2.00) node [anchor=west] {$0.000$};
\draw (0.5, 2.00) node [anchor=west] {$0.002$};
\end{tikzpicture}
\end{minipage}

\begin{minipage}{.02\linewidth}
\begin{turn}{90}
Estimated error
\end{turn}

\begin{center}
$ $
\end{center}

\end{minipage}
\begin{minipage}{.30\linewidth}
\begin{tikzpicture}
\draw (0,0) node[anchor=south west,inner sep=0]%
{\includegraphics[width=4.0cm]{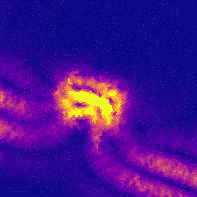}};
\draw[thick] (0,0) rectangle (4.0,4.0);
\end{tikzpicture}

\begin{center}
$t = 4.0$
\end{center}
\end{minipage}
\begin{minipage}{.30\linewidth}
\begin{tikzpicture}
\draw (0,0) node[anchor=south west,inner sep=0]%
{\includegraphics[width=4.0cm]{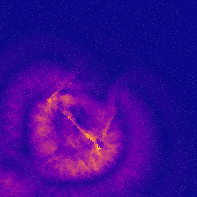}};
\draw[thick] (0,0) rectangle (4.0,4.0);
\end{tikzpicture}

\begin{center}
$t = 5.5$
\end{center}

\end{minipage}
\begin{minipage}{.30\linewidth}
\begin{tikzpicture}
\draw (0,0) node[anchor=south west,inner sep=0]%
{\includegraphics[width=4.0cm]{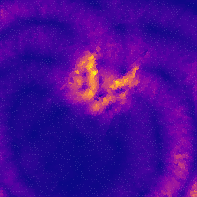}};
\draw[thick] (0,0) rectangle (4.0,4.0);
\end{tikzpicture}

\begin{center}
$t = 6.5$
\end{center}

\end{minipage}
\begin{minipage}{.05\linewidth}
\begin{tikzpicture}
\draw (0,-2.00) node[anchor=south west,inner sep=0]%
{\includegraphics[width=0.5cm,height=4.0cm]{figures/reflection/images/cbar_plasma}};
\draw[thick] (0,-2.00) rectangle (.5,2.00);
\draw[thick] (0,-2.00) rectangle (.5,2.00);

\draw (0.5,-2.00) node [anchor=west] {$0.000$};
\draw (0.5, 2.00) node [anchor=west] {$0.002$};
\end{tikzpicture}

\begin{center}
$ $
\end{center}

\end{minipage}

\caption{Obstacle example with $\sigma = 0.25$ and $h = 0.0125$}
\label{figure_penetrable_pictures}
\end{figure}